\documentclass[11pt]{amsart}

\usepackage{mathtools}
\usepackage{amsmath, amsthm, amssymb, amsfonts, enumerate}
\usepackage{dsfont}
\usepackage{color}
\usepackage{geometry}
\usepackage{todonotes}
\usepackage{graphicx}
\usepackage{caption}
\usepackage{float}
\usepackage{soul}
\usepackage[colorlinks=true,linkcolor=blue,urlcolor=blue, citecolor=blue]{hyperref}

\def \cA{\mathcal{A}}

\def \cC{\mathcal{C}}

\def \cF{\mathcal{F}}

\def \cH{\mathcal{H}}

\def \cJ{\mathcal{J}}

\def \cL{\mathcal{L}}

\def \cO{\mathcal{O}}

\def \cT{\mathcal{T}}
\def \P{\mathsf P}

\def \E{\mathsf E}

\def \N{\mathbb{N}}
\def \R{\mathbb{R}}
\def \F{\mathbb F}

\def \ud{\mathrm{d}}

\def \lbracket{(}

\newcommand{\eps}{\varepsilon}

\newtheorem{theorem}{Theorem}[section]
\newtheorem{lemma}[theorem]{Lemma}
\newtheorem{corollary}[theorem]{Corollary}
\newtheorem{proposition}[theorem]{Proposition}

\newtheorem{remark}[theorem]{Remark}
\newtheorem{assumption}[theorem]{Assumption}

\theoremstyle{definition}

\DeclareMathOperator{\sign}{sign}

\definecolor{ballblue}{rgb}{0.13, 0.67, 0.8}

\makeatletter
\@namedef{subjclassname@2020}{\textup{2020}Mathematics Subject Classification}
\makeatother

\title[Stopper-(singular)controller games with constrained controls]{Zero-sum stopper vs.\ singular-controller games \\ with constrained control directions}
\author[Bovo]{Andrea Bovo}
\author[De Angelis]{Tiziano De Angelis}
\author[Palczewski]{Jan Palczewski}
\address{A.\ Bovo: School of Management and Economics, Dept.\ ESOMAS, University of Torino, Corso Unione Sovietica, 218 Bis, 10134, Torino, Italy; Collegio Carlo Alberto, Piazza Arbarello 8, 10122, Torino, Italy.}
\email{\href{mailto:andrea.bovo@unito.it}{andrea.bovo@unito.it}}
\address{T.\ De Angelis: School of Management and Economics, Dept.\ ESOMAS, University of Torino, Corso Unione Sovietica, 218 Bis, 10134, Torino, Italy; Collegio Carlo Alberto, Piazza Arbarello 8, 10122, Torino, Italy.}
\email{\href{mailto:tiziano.deangelis@unito.it}{tiziano.deangelis@unito.it}}
\address{J.\ Palczewski: School of Mathematics, University of Leeds, Woodhouse Lane, LS2 9JT Leeds, UK.}
\email{\href{mailto:j.palczewski@leeds.ac.uk}{j.palczewski@leeds.ac.uk}}

\date{\today}
\numberwithin{equation}{section}

\begin{document}
\begin{abstract}
We consider a class of zero-sum stopper vs.\ singular-controller games in which the controller can only act on a subset $d_0<d$ of the $d$ coordinates of a controlled diffusion. Due to the constraint on the control directions these games fall outside the framework of recently studied variational methods. In this paper we develop an approximation procedure, based on $L^1$-stability estimates for the controlled diffusion process and almost sure convergence of suitable stopping times. That allows us to prove existence of the game's value and to obtain an optimal strategy for the stopper, under continuity and growth conditions on the payoff functions.
This class of games is a natural extension of (single-agent) singular control problems, studied in the literature, with similar constraints on the admissible controls.
\end{abstract}

\maketitle

\section{Introduction}

A zero-sum stopper vs.\ singular-controller game can be formulated as follows. Given a time horizon $T\in(0,\infty)$, two players observe a stochastic dynamics $X=(X_s)_{s\in[0,T]}$ in $\R^d$ described by a controlled stochastic differential equation (SDE). One player (the minimiser) may exert controls that impact additively on the dynamics and that may be singular with respect to the Lebesgue measure, as functions of time. The other player (the maximiser) decides when the game ends by selecting a stopping time in $[0,T]$. At the end of the game, the first player (controller) pays the second one (stopper) a payoff that depends on time, on the sample paths of $X$ and on the amount of control exerted. A natural question is whether the game admits a value, i.e., if the same expected payoff is attained irrespective of the order in which the players choose their (optimal) actions.

In \cite{bovo2022variational} we studied zero-sum stopper vs.\ singular-controller games in a diffusive setup with controls that can be exerted in all $d$ coordinates of the process $X$. The approach is based on a mix of probabilistic and analytic methods for the study of a class of variational inequalities with so-called obstacle and gradient constraints. It is shown that the value of the game is the maximal solution of such variational inequality. More precisely, it is the maximal {\em strong solution} in the sense that it belongs to the Sobolev space of functions that admit two spatial derivatives and one time derivative, locally in $L^p$ (i.e., in $W^{1,2,p}_{\ell oc}$). The methods rely crucially on the assumption that {\em all} coordinates of the process can be controlled. Indeed, that determines a particular form of the gradient constraint that enables delicate PDE estimates for a-priori bounds on the solution. When only $d_0<d$ coordinates are controlled, i.e., there is a constraint on the control directions, the results from \cite{bovo2022variational} are not applicable (cf.\ Section \ref{sec:challenge} for details) and the existence of a value is an open question.

In this paper, we continue our study of zero-sum stopper vs.\ singular-controller games by showing that even in the case $d_0<d$ the game admits a value. We also provide an optimal strategy for the stopper and we observe that it is of a slightly different form compared to the one obtained in \cite{bovo2022variational} (see Remark \ref{rem:theta} below for details). The line of proof follows an approximation procedure, governed by a parameter $\gamma\in[0,1]$, by which we relax the constraints on the class of admissible controls. For $\gamma=1$ we are in the same setting as in \cite{bovo2022variational}, whereas $\gamma=0$ corresponds to the constrained case. It turns out that for $\gamma\in(0,1)$ we have an intermediate situation for which a suitable adaptation of the arguments from \cite{bovo2022variational} is possible. The idea is then to obtain the value of the constrained game in the limit as $\gamma\downarrow 0$. 

When letting $\gamma\downarrow 0$, we need $L^1$-stability estimates for the controlled dynamics. These estimates involve local times and a-priori bounds on the candidate optimal controls and they are not standard in the literature. Optimality of the stopper's strategy is derived via an almost sure convergence for a suitable sequence of stopping times, based on path properties of the controlled dynamics and uniform convergence of the approximating value functions as $\gamma\downarrow 0$. We can no longer guarantee the solvability of the associated variational problem in the strong (Sobolev) sense but, of course, our value function satisfies both the appropriate gradient constraint and obstacle constraint. Moreover, we show that the value of our game is the uniform limit of solutions of approximating variational inequalities, paving the way to a notion of solution in the viscosity sense. Finally, we notice that our results hold under continuity and (sub)linear growth conditions on the payoff functions. These are much weaker conditions than those needed in \cite{bovo2022variational}, where continuous differentiability in time and space and H\"older continuity of the derivatives is required.

The motivation for considering constrained control directions arises from the literature on (single-agent) irreversible or partially reversible investment problems. In the classical paper \cite{soner1991free}, Soner and Shreve consider a $d$-dimensional Brownian motion whose $d$-th coordinate is singularly controlled. Various works by Zervos et al. (e.g., \cite{merhi2007model,lokka2011model,lokka2013long}), Guo and Tomecek \cite{guo2009class}, Federico et al. (e.g., \cite{federico2014characterization,federico2021singular}, Ferrari (e.g., \cite{ferrari2015integral}), De Angelis et al. (e.g., \cite{de2017optimal,de2015nonconvex,deangelis2019solvable}) consider 2- or 3-dimensional dynamics with only one controlled coordinate. We also notice that in those papers the controlled process $X$ is fully degenerate in the controlled dynamics (i.e., there is no diffusion in the control direction). In all cases but \cite{de2017optimal} and \cite{federico2021singular} this assumption enables an explicit solution of the problem, because the resulting free boundary problems are cast as families of ODEs parametrised by the state variable associated to the control. A non-degenerate example arises instead in mathematical finance in the paper by Bandini et al.\ \cite{bandini2022optimal} who deal with a 2-dimensional diffusive dynamics with only one controlled coordinate. It seems therefore natural that game versions of similar problems should be studied in detail and we provide the first results in this direction.

The literature on controller vs.\ stopper games has been developing in various directions in the case of controls with {\em bounded velocity} (see, e.g., Bensoussan and Friedman \cite{bensoussan1974nonlinear}, Karatzas et al. \cite{karatzas2001controller,karatzas2008martingale}, Hamadene \cite{hamadene2006stochastic}, Bayraktar and Li \cite{bayraktar2019controller}, among others). A more detailed review of the main results in that direction is provided in the introduction of \cite{bovo2022variational}. Instead, the case of singular controls is widely unexplored. Prior to \cite{bovo2022variational}, the only other contribution was by Hernandez-Hernandez et al.\ \cite{hernandez2015zero} (see also \cite{hernandez2015zsgsingular}), who studied the problem in a one-dimensional setting using free boundary problems in the form of ODEs with appropriate boundary conditions. The present paper contributes to the systematic study of zero-sum stopper vs.\ singular controller games while complementing and extending the classical framework with controls of bounded velocity. 

Our paper is organised as follows. In Section \ref{sec:setting} we set up the problem, we explain the main technical difficulties preventing the use of methods from \cite{bovo2022variational}, we state the main result (Theorem \ref{thm:usolvar}) and introduce an approximation scheme. In Section \ref{sec:approx} we obtain stability estimates. Those are later used in Section \ref{sec:valgame} to prove convergence of the value functions of the approximating problems to the original one. A technical appendix completes the paper.


\section{Setting and main results}\label{sec:setting}
Let $(\Omega,\cF,\P)$ be a complete probability space, equipped with a right-continuous filtration $\F = (\cF_s)_{s\in[0,\infty)}$ completed with $\P$-null sets. Let $(W_s)_{s\in[0,\infty)}$ be an $\F$-adapted, $d'$-dimensional Brownian motion. Fix $T\in(0,\infty)$, the horizon of the game. Let $d \le d'$ be the dimension of the controlled diffusion process $(X_s)_{s\in[0,T]}$. We decompose $d$ into two sets of coordinates: $d = d_0 + d_1$ with $d_0, d_1 > 0$. The first $d_0$ coordinates in the controlled dynamics are affected directly by singular controls. The remaining $d_1$ coordinates, instead, are affected indirectly via drift and diffusion coefficients. This is made rigorous in \eqref{eq:prcXcntrll} after we introduce the class of admissible controls.

For $t\in[0,T]$, we denote
\begin{align*}
\cT_t\coloneqq \left\{\tau|\tau\text{ is a stopping time such that}\ \tau\in[0, T-t]\right\}.
\end{align*}
For a vector $x \in \R^d$, $|x|_d$ stands for the Euclidean norm of $x$ and $|x|_{d_0}$ for the Euclidean norm of the first $d_0$ coordinates. We consider the following class of admissible controls
\begin{align*}
 \cA^{d_0}\coloneqq \left\{(n,\nu)\left|
\begin{aligned}
&\text{$(n_s)_{s\in[0,T]}$ is progressively measurable, $\R^{d}$-valued,} \\
&\text{with $n_s=(n^1_s,\ldots, n_s^{d_0},0,0,\ldots, 0)$, $\forall s\in[0,T]$ }\\
&\text{and $|n_s|_{d}=|n_s|_{d_0}=1$, $\P$-a.s.\ $\forall s\in[0,T]$};\\
&\text{$(\nu_s)_{s\in[0,T]}$ is $\mathbb{F}$-adapted, real valued, non-decreasing and}\\
&\,\text{right-continuous with $\nu_{0-}=0$, $\P\text{-a.s.}$, and $\E[(\nu_{T})^2]<\infty$}
\end{aligned}%
\right.\right\}.
\end{align*}
Analogously, we define the class $\cA^d$ with the same properties as the one above but with $n_s=(n^1_s,\ldots, n^d_s)$ such that $|n_s|_d=1$, $\P$-a.s. The class $\cA^d$ is the one used by \cite{bovo2022variational}, where the control may act in all $d$ coordinates. Instead, the class $ \cA^{d_0}$ is the one which we use in the present paper, where {\em the control directions are constrained} to a subspace of $\R^d$. 

Notice that for $\P$-a.e.\ $\omega$, the map $s\mapsto n_s(\omega)$ is Borel-measurable on $[0,T]$ and $s\mapsto \nu_s(\omega)$ defines a measure on $[0,T]$; thus the Lebesgue-Stieltjes integral $\int_{[0,s]}n_u(\omega)\ud \nu_u(\omega)$ is well-defined for $\P$-a.e.\ $\omega$. A jump of the process $\nu$ at time $s$ is denoted by $\Delta \nu_s\coloneqq\nu_s-\nu_{s-}$.

Given a control pair $(n,\nu)\in \cA^{d_0}$ and an initial condition $x\in\R^d$, we consider a $d$-dimensional controlled stochastic dynamics $(X_s^{[n,\nu]})_{s\in[0,T]}$ described by 
\begin{align}\label{eq:prcXcntrll}
\ud X_{s}^{[n,\nu]}=&\, b(X_s^{[n,\nu]})\,\ud s+ \kappa(X_s^{[n,\nu]})\ud W_s+ n_{s}\,\ud\nu_s,\quad X^{[n,\nu]}_{0-}=x,
\end{align}
where $b:\R^d\to\R^d$ and $\kappa:\R^{d}\to \R^{d\times d'}$ are continuous functions and $X^{[n,\nu]}_{0-}$ is the state of the dynamics before a possible jump at time zero. 
We denote 
\[
\P_x\big(\,\cdot\,\big)=\P\big(\,\cdot\,\big|X^{[n,\nu]}_{0-}=x\big)\quad\text{and}\quad\E_x\big[\,\cdot\,\big]=\E\big[\,\cdot\,\big|X^{[n,\nu]}_{0-}=x\big].
\]
It is important to remark that the control acts only in the first $d_0$ coordinates of the dynamics of $X^{[n,\nu]}$. However, the effect of such control is also felt by the remaining $d_1$ coordinates via the drift and diffusion coefficients. Under Assumption \ref{ass:gen1} on $b$ and $\kappa$ (stated below), there is a unique (strong) $\F$-adapted solution of \eqref{eq:prcXcntrll} by, e.g., \cite[Thm.\ 2.5.7]{krylov1980controlled}.

We study a class of 2-player zero-sum games (ZSGs) between a (singular) controller and a stopper. The stopper picks $\tau\in\cT_t$ and the controller chooses a pair $(n,\nu)\in \cA^{d_0}$. At time $\tau$ the game ends and the controller pays to the stopper a random payoff depending on $\tau$ and on the path of $X^{[n,\nu]}$ up to time $\tau$. We denote the state space of the game by
\[
\R^{d+1}_{0,T}\coloneqq[0,T]\times\R^d.
\]
Consider continuous functions $g,h:\R^{d+1}_{0,T}\to [0,\infty)$, $f:[0,T]\to(0,\infty)$, and a fixed discount rate $r\ge 0$ be given. For $(t,x)\in\R^{d+1}_{0,T}$, $\tau \in \cT_t$ and $(n, \nu) \in  \cA^{d_0}$, the game's {\em expected} payoff reads
\begin{align}\label{eq:payoff}
\cJ_{t,x}(n,\nu,\tau)= \E_{x}\Big[e^{-r\tau}\!g(t\!+\!\tau,X_\tau^{[n,\nu]})\!+\!\int_0^{\tau}\!\! e^{-rs}h(t\!+\!s,X_s^{[n,\nu]})\,\ud s\! +\!\int_{[0,\tau]}\!\! e^{-rs}f(t\!+\!s)\,\ud \nu_s \Big].
\end{align}

We define the {\em lower} and {\em upper} value of the game respectively by
\begin{align}\label{eq:lowuppvfnc_1}
\underline{v}(t,x)\coloneqq \adjustlimits\sup_{\tau\in \mathcal{T}_t}\inf_{(n,\nu)\in \cA^{d_0}} \cJ_{t,x}(n,\nu,\tau)\quad\text{and}\quad\overline{v}(t,x)\coloneqq \adjustlimits\inf_{(n,\nu)\in \cA^{d_0}}\sup_{\tau\in \mathcal{T}_t} \cJ_{t,x}(n,\nu,\tau).
\end{align}
Then $\underline{v}(t,x)\leq \overline{v}(t,x)$ and if the equality holds we say that the game \emph{admits a value}: 
\begin{align}\label{eq:valuegame_1}
v(t,x)\coloneqq \underline{v}(t,x)=\overline{v}(t,x).
\end{align}

Before assumptions of the paper are formulated, we introduce necessary notations. Given a matrix $M\in\R^{d\times d'}$, with entries $M_{ij}$, $i=1,\ldots, d$, $j=1,\ldots, d'$, we define its norm by
\begin{align*}
|M|_{d\times d'}\coloneqq \Big(\sum_{i=1}^d\sum_{j=1}^{d'}M_{ij}^2\Big)^{1/2},
\end{align*}
and, if $d=d'$, we let $\mathrm{tr}(M)\coloneqq \sum_{i=1}^d M_{ii}$. For $x\in\R^d$ we use the notation $x=(x_{[d_0]}, x_{[d_1]})$ with $x_{[d_0]}=(x_1,\ldots, x_{d_0})$ and $x_{[d_1]}=(x_{d_0+1},\ldots, x_{d})$. Given a smooth function $\varphi:\R^{d+1}_{0,T}\to \R$ we denote its partial derivatives by 
$\partial_t \varphi$, $\partial_{x_i}\varphi$, $\partial_{x_ix_j}\varphi$, for $i,j=1,\ldots, d$. We write $\nabla \varphi=(\partial_{x_1}\varphi,\ldots, \partial_{x_d} \varphi)$ for the spatial gradient, and $D^2 \varphi = (\partial_{x_i x_j}\varphi)_{i,j=1}^d$ for the spatial Hessian matrix. The first $d_0$ coordinates of the gradient $\nabla \varphi$ are denoted by $\nabla^0 \varphi=(\partial_{x_1}\varphi,\ldots, \partial_{x_{d_0}}\varphi)$ and the remaining $d_1$ coordinates are denoted by $\nabla^{1}\varphi=(\partial_{x_{d_0+1}}\varphi,\ldots, \partial_{x_{d}}\varphi)$.

We now give assumptions under which we obtain our main result, Theorem \ref{thm:usolvar}.
\begin{assumption}[Controlled SDE]\label{ass:gen1}
The functions $b$ and $\kappa$ are such that:
\begin{itemize}
\item[(i)] $b\in C^1(\R^d;\R^d),$ $\kappa\in C^1(\R^d;\R^{d\times d'})$ with derivatives bounded by $D_1>0$;
\item[(ii)] For $i=1,\ldots, d$ and $\kappa^i=(\kappa_{i1},\ldots, \kappa_{id'})$, it holds $\kappa^i(x)=\kappa^i(x_i)$;
\item[(iii)] For any bounded set $B\subset\R^d$ there is $\theta_B>0$ such that 
\begin{align}\label{eq:defnthetaEC_1}
\langle\zeta,\kappa\kappa^\top (x)\zeta\rangle\geq\theta_B|\zeta|_d^2\,\qquad\text{for any $\zeta\in\R^d$ and all $x\in\overline{B}$,}
\end{align}
where $\langle\cdot, \cdot\rangle$ denotes the scalar product in $\R^d$ and $\overline{B}$ the closure of $B$.
\end{itemize}
\end{assumption}
Notice that the Lipschitz continuity of $b$ and $\kappa$ implies that there exists $D_2$ such that
\begin{align}\label{eq:lingrowcff21}
|b(x)|_d+|\kappa(x)|_{d\times d'}\leq D_2(1+|x|_d), \qquad\text{for all $x\in\R^d$.}
\end{align}

\begin{assumption}[Payoffs]\label{ass:gen2}
Functions $f:[0,T]\to (0,\infty)$, $g,h:\R^{d+1}_{0,T}\to [0,\infty)$ are continuous, and:
\begin{itemize}
	\item[(i)] The function $f$ is non-increasing;
	\item[(ii)] There exist constants $K_1\in(0,\infty)$ and $\beta\in[0,1)$ such that
	\[
	0\leq g(t,x)+h(t,x)\leq K_1(1+|x|^\beta_d)\quad\text{for all $(t,x)\in\R^{d+1}_{0,T}$;}
	\]
	\item[(iii)] The function $g$ is Lipschitz in the first $d_0$ spatial coordinates with a constant bounded by $f$ in the sense that for every $t \in [0, T]$, $|\nabla^{0}g(t,x)|_{d_0}\le f(t)$ for a.e.\ $x\in\R^{d}$.
	\end{itemize}
\end{assumption}
Assumption \ref{ass:gen1}(ii) says that the diffusion coefficient of each coordinate of the process depends only on such coordinate. That is needed for $L^1$-stability estimates provided in Sections \ref{sec:stability} and it is satisfied by, e.g., stock market models with stochastic interest rates (cf.\ \cite{cai2022american}). The assumptions on the payoff functions are in line with those in \cite{bovo2022variational}. More precisely, we allow less smoothness than in \cite{bovo2022variational} but we require strictly sub-linear growth instead of quadratic growth as in \cite{bovo2022variational}. Those assumptions are satisfied by a wide class of strictly concave utility functions. Finally, we could allow for $r< 0$ by incorporating the discount factor in functions $f,g,h$ (which are time-dependent).

The next theorem is the main result of the paper. Its proof builds on an approximation procedure that allows us to invoke PDE results from \cite{bovo2022variational}. By passing to the limit in the approximation scheme we recover the value function of our game. Details of the scheme and the convergence are presented in the next sections of the paper.
\begin{theorem}\label{thm:usolvar}
Under Assumptions \ref{ass:gen1} and \ref{ass:gen2}, the game described above admits a value $v$ (i.e., \eqref{eq:valuegame_1} holds) with the following properties:
\begin{itemize}
\item[ (i)] $v$ is continuous on $\R^{d+1}_{0,T}$; 
\item[ (ii)] $|v(t,x)|\le c(1+|x|_d^\beta)$ for some $c>0$, and $\beta$ from Assumption \ref{ass:gen2}(ii); 
\item[(iii)] $v$ is Lipschitz continuous in the first $d_0$ spatial variables with constant bounded by $f$ in the sense that $|\nabla^0 v(t,x)|_{d_0}\le f(t)$ for a.e.\ $(t,x)\in\R^{d+1}_{0,T}$. 
\end{itemize}

Moreover, for any given $(t,x)\in\R^{d+1}_{0,T}$ and any $(n,\nu)\in \cA^{d_0}$, the stopping time $\theta_*=\theta_*(t,x;n,\nu)\in\cT_t$
is optimal for the stopper, where $\theta_*=\tau_*\wedge\sigma_*$ and $\P_x$-a.s.
\begin{align}\label{eq:taustar_1}
\begin{split}
&\tau_*=\tau_*(t,x;n,\nu)\coloneqq \inf\big\{s\geq 0\,\big|\, v(t+s,X_s^{[n,\nu]})=g(t+s,X_s^{[n,\nu]})\big\},\\
&\sigma_*=\sigma_*(t,x;n,\nu)\coloneqq \inf\big\{s\geq 0\,\big|\, v(t+s,X_{s-}^{[n,\nu]})=g(t+s,X_{s-}^{[n,\nu]})\big\}.
\end{split}
\end{align}
\end{theorem}

\begin{remark}
The set $\{s\geq0\,|\,v(t+s,X_s^{[n,\nu]})=g(t+s,X_s^{[n,\nu]})\big\}$ always contains $T-t$, because $v(T, x) = g(T,x)$. Hence, $\theta_*\le T-t$, $\P_x$-a.s.
\end{remark}

\begin{remark}\label{rem:opt}
The stopper's strategy $\theta_*$ is of a closed-loop type, i.e., the stopping time $\theta_*$ depends on the dynamics of the underlying process $X^{[n, \nu]}$. Optimality of $\theta_*$, asserted above, should be understood in the sense that for any admissible control $(n,\nu)\in \cA^{d_0}$, we have
\[
v(t,x) \le \cJ_{t,x}(n,\nu, \theta_*), \qquad (t,x) \in \R^{d+1}_{0,T}.
\]
In particular this implies $v(t,x)=\inf_{(n,\nu)\in\cA^{d_0}}\cJ_{t,x}(n,\nu,\theta_*(t,x; n,\nu))$, but $\theta_*$ may not be a best response for any specific pair $(n,\nu)$. Existence of an optimal control remains an open question (cf.\ \cite[Rem.\ 3.5]{bovo2022variational}).
\end{remark}

\begin{remark}\label{rem:dd0}
The results in the theorem above continue to hold in the unconstrained case $d_0=d$. That proves existence of a value under less stringent regularity conditions on $g,h$ than in \cite{bovo2022variational} and when $f$ is independent of the spatial coordinate. Notice that for $d=d_0$ the approximation via functions $(u^\gamma)_{\gamma>0}$ described in Section \ref{sec:penprobg} is not needed. The rest of the analysis follows the same steps as in Section \ref{sec:valgame} taking $\gamma=1$ and ignoring the arguments about the limit as $\gamma\to 0$ in Section \ref{sec:gammato0}.
\end{remark}

\begin{remark}\label{rem:theta}
The stopping time $\tau_*$ is shown to be optimal for the game studied in \cite{bovo2022variational}. Theorem \ref{thm:usolvar} asserts the optimality of $\theta_*$ which is the minimum of $\tau_*$ and another stopping time $\sigma_*$. This construction comes at no disadvantage as $\theta_*$ is also optimal in the setting of \cite{bovo2022variational} (see Lemma \ref{lem:convst}). It however enables us to prove convergence of optimal stopping times in the form of $\theta_*$ for games with value functions converging uniformly on compacts (see Lemma \ref{lem:convth} and the proof of Theorem \ref{thm:opttaustar_1}). Note that one cannot expect such convergence to hold for $\tau_*$.
\end{remark}

\subsection{Approximation procedure}\label{sec:penprobg}
The key step for the proof of Theorem \ref{thm:usolvar} is based on an approximation scheme that we present here. Fix $\gamma\in(0,1]$
. Given $(n,\nu)\in\cA^d$, we consider the controlled SDE
\begin{align}\label{eq:SDEgam}
\ud X_{s}^{[n,\nu],\gamma}=&\, b(X_s^{[n,\nu],\gamma})\,\ud s+ \kappa(X_s^{[n,\nu],\gamma})\ud W_s+ n^\gamma_{s}\,\ud\nu_s,
\end{align}
where $n^\gamma_s\coloneqq(n^1_s,\ldots,n^{d_0}_s,\gamma n^{d_0+1}_s,\ldots,\gamma n^{d}_s)$ (i.e., the parameter $\gamma$ acts as a weight on the last $d_1$ coordinates of $n_s$).

Given vectors $p,q\in \R^{d}$, recalling the notation $p=(p_{[d_0]},p_{[d_1]})\in\R^{d_0}\times \R^{d_1}$ and the scalar product $\langle p,q\rangle$ in $\R^d$, we introduce the bilinear form $\langle\cdot,\cdot\rangle_{\gamma}:\R^d\times\R^d\to \R$ defined as $\langle p,q\rangle_\gamma \coloneqq \langle p_{[d_0]},q_{[d_0]}\rangle+\gamma\langle p_{[d_1]},q_{[d_1]}\rangle$.
Notice that we are slightly abusing the notation because $\langle p_{[d_0]},q_{[d_0]}\rangle$ and $\langle p_{[d_1]},q_{[d_1]}\rangle$ are scalar products in $\R^{d_0}$ and $\R^{d_1}$, respectively.
Associated with $\langle\cdot,\cdot\rangle_\gamma$ we have the norm 
$|p|_\gamma \coloneqq \sqrt{\langle p,p\rangle_\gamma}$ on $\R^d$.
It is worth noticing that 
$\nabla |p|^2_\gamma=2 (p_1,\ldots, p_{d_0},\gamma\, p_{d_0+1},\ldots, \gamma\, p_d)$
and, for $j=1,\ldots, d$, we clearly have $\big(D^2|p|^2_\gamma\big)_{ij}=2\delta_{ij}$ for $i=1,\ldots, d_0$ and $\big(D^2|p|^2_\gamma\big)_{ij}=2\gamma\delta_{ij}$ for $i=d_0+1,\ldots, d$, where $\delta_{ij}$ is the Kronecker delta. 

We introduce an approximation of $f$ as
\begin{align}\label{eq:defnfgamma}
f^\gamma(t) \coloneqq \sqrt{f^2(t)+\gamma K^2}\quad\text{for $t\in[0,T]$},
\end{align}
where $K$ is a suitable constant that we choose later on (the same as in \eqref{eq:ghlipL_1}). By construction $f^\gamma\to f$ uniformly on $[0,T]$ as $\gamma\to0$. We consider a new payoff
\begin{align*}
\cJ_{t,x}^\gamma(n,\nu,\tau) &\coloneqq \E_{x}\Big[e^{-r\tau} g(t\!+\!\tau,X_\tau^{[n,\nu],\gamma}) + \int_0^{\tau}\! e^{-rs}h(t\!+\!s,X_s^{[n,\nu],\gamma}) \ud s\\
&\hspace{30pt}+ \int_{[0,\tau]}\!\!\! e^{-rs}f^\gamma(t\!+\!s)\ud \nu_s\! \Big].
\end{align*}
Upper and lower value for the game with the expected payoff $\cJ^\gamma_{t,x}$ are given by
\begin{align*}
\underline u^\gamma(t,x)= \sup_{\tau\in\cT_t}\inf_{(n,\nu)\in\cA^d}\cJ_{t,x}^\gamma(n,\nu,\tau)\quad\text{and}\quad \overline u^\gamma(t,x)= \!\inf_{(n,\nu)\in\cA^d}\sup_{\tau\in\cT_t}\cJ_{t,x}^\gamma(n,\nu,\tau),
\end{align*}
and we say that the value exists if $u^\gamma\coloneqq\underline u^\gamma=\overline u^\gamma$. We formally set $\overline u^0=\overline v$ and $\underline u^0=\underline v$.

Using results from \cite{bovo2022variational} we will show that $u^\gamma$ is well-defined, i.e., the approximating game admits a value, for every $\gamma\in(0,1]$. Then we obtain $\lim_{\gamma \to 0}u^\gamma=\overline v$ and $\lim_{\gamma \to 0}u^\gamma=\underline v$ uniformly on compacts, thus proving existence of a value for our constrained game.

\subsection{Challenges in the constrained setup}\label{sec:challenge}
The theory developed in \cite{bovo2022variational} does not cover the game we are considering here for two essential reasons. The first one is that the functions $f,g,h$ are only assumed to be continuous, whereas \cite{bovo2022variational} requires continuous differentiability once in time and twice in space (and H\"older continuity of all derivatives). The second one, and more important, is that the constraints on the directions of the admissible control imply that estimates obtained in \cite{bovo2022variational} via analytical arguments can no longer be obtained. In the next paragraphs we briefly elaborate on this fine technical issue.

The variational problem in \cite{bovo2022variational} features a gradient constraint on the value function $v$ of the form $|\nabla v|_d\le f$. In the penalisation procedure adopted in \cite{bovo2022variational} we therefore consider a semi-linear PDE with a non-linear term of the form $\R^d\ni p\mapsto\psi_\eps(|p|_d^2-f^2)$ (see Eq. (5.14) in \cite{bovo2022variational}), where $\eps>0$ is a parameter that must tend to zero in the limit of the penalisation step. In our current setup, given that the control only acts in the first $d_0$ coordinates, the gradient constraint must be of the form $|\nabla^0 v|_{d_0}\le f$. That translates into a non-linear term of the form $\R^d\ni p\mapsto\psi_\eps(|p_{[d_0]}|_{d_0}^2-f^2)$ in the associated penalised problem.
One of the key estimates in \cite{bovo2022variational} is obtained in \cite[Prop.\ 4.9]{bovo2022variational} and it concerns a bound on the gradient of the solution of the penalised problem. The method of proof adopted in \cite[Prop.\ 4.9]{bovo2022variational} is also used in other places, e.g., in \cite[Prop.\ 5.1]{bovo2022variational}. We now show where those arguments fail. 

Arguing as in the proof of \cite[Prop.\ 4.9]{bovo2022variational}, we obtain an analogue of \cite[Eq.\ \lbracket 4.38)]{bovo2022variational} and it reads:
\begin{align*}
-2\langle \nabla w^n,\nabla (|\nabla^{0} u^n|^2_{d_0}- f_m^2)\rangle\leq -2\lambda |\nabla^0 u^{\eps,\delta}|_{d_0}^2+\tilde{R}_n.
\end{align*}
Above, it is enough to understand that $f_m$ is an approximation of $f$, while $w^n$ and $u^n$ both approximate the solution $u^{\eps,\delta}$ of the penalised problem.
The term $\tilde{R}_n$ is a remainder which can be made arbitrarily small and it plays no substantial role in this discussion. Continuing with the argument that follows \cite[Eq.~\lbracket 4.38)]{bovo2022variational} we arrive at
\[
\lambda|\nabla^0 u^{\eps,\delta}|^2_{d_0}\le \alpha_1 |\nabla u^{\eps,\delta}|^2_{d}+\alpha_2,
\]
where $\alpha_1,\alpha_2>0$ are given constants and $\lambda>0$ can be chosen arbitrarily. From this estimate we cannot conclude that $|\nabla u^{\eps,\delta}|$ is bounded. Instead of $\lambda|\nabla^0 u^{\eps,\delta}|^2_{d_0}$, in \cite{bovo2022variational} we have $\lambda|\nabla u^{\eps,\delta}|^2_{d}$, which leads to $\lambda|\nabla u^{\eps,\delta}|^2_{d}\le \alpha_1 |\nabla u^{\eps,\delta}|^2_{d}+\alpha_2$ and it allows to conclude $|\nabla u^{\eps,\delta}|^2_{d}\le c$ for some constant $c>0$, by the arbitrariness of $\lambda$.

Other difficulties of a similar nature appear in, e.g., adapting the arguments of \cite[Lem.\ 5.8]{bovo2022variational}, where in Eq.~(5.34) we would not be able to obtain a bound on $|D^2w^n|_{d\times d}^2$, because we cannot control the derivatives $\partial_{x_ix_j}w^n$ for $i,j=d_0+1,\ldots, d$. We avoid going into further detail and refer the interested reader to the original paper for a careful comparison. 

\subsection{Notation}\label{sec:pre}
Before passing to the proof of Theorem \ref{thm:usolvar}, we introduce the remaining notation used in the paper.

The $d$-dimensional {\em open} ball centred in $0$ with radius $m$ is denoted by $B_m$. For an arbitrary subset $\cO\subseteq \R^{d+1}_{0,T}$ we let $C^\infty_{c,\, \rm sp}(\cO)$ be the space of functions on $\cO$ {\em with compact support in the spatial coordinates} (not in time) and infinitely many continuous derivatives. For an open bounded set $\cO\subset\R^{d+1}_{0,T}$, we denote by $\overline\cO$ the closure of $\cO$ and we let $C^0(\overline{\cO})$ be the space of continuous functions $\varphi:\overline\cO\to\R$ equipped with the supremum norm
\begin{align}\label{eq:supremum}
\|\varphi\|_{C^0(\overline\cO)}\coloneqq \sup_{(t,x)\in\overline \cO}|\varphi(t,x)|.
\end{align}
Analogously, $C^0(\R^{d+1}_{0,T})$ is the space of bounded and continuous functions $\varphi:\R^{d+1}_{0,T}\to\R$ equipped with the norm $\|\varphi\|_\infty\coloneqq \|\varphi\|_{C^0(\R^{d+1}_{0,T})}$.
We denote by $C^{0,1,\alpha}(\overline\cO)$ the space of $\alpha$-H\"older continuous functions on $\overline \cO$ with $\alpha$-H\"older continuous spatial gradient, equipped with the supremum norm and the $\alpha$-H\"older semi-norm. The semi-norm is evaluated with respect to the parabolic distance; for details see \cite[Ch.\ 3, Sec.\ 2]{friedman2008partial} (see also the notation section in \cite{bovo2022variational}). The space of functions with bounded $C^{0,1,\alpha}$-norm in any compact subset of $\R^{d+1}_{0,T}$ is denoted by $C^{0,1,\alpha}_{\ell oc}(\R^{d+1}_{0,T})$. For $p\in[1,\infty)$, we recall the definition of the usual Sobolev space (see \cite[Sec.\ 2.2]{krylov2008lectures}):
\begin{align}\label{eq:W12p}
W^{1,2,p}_{\ell oc}(\R^{d+1}_{0,T})&\,\coloneqq \big\{ f\in L^p_{\ell oc}(\R^{d+1}_{0,T}) \,\big|\,f\in W^{1,2,p}(\cO), \,\forall \cO\subseteq \R^{d+1}_{0,T}, \cO\text{ bounded}\big\}.
\end{align}

For $e_1=(1,0,\ldots,0)\in\R^d$ and $a(x) \coloneqq (\kappa\kappa^\top)(x)$, the infinitesimal generator of the uncontrolled process $X^{[e_1,0]}$ is denoted by $\mathcal{L}$ and it reads
\begin{align*}
(\mathcal{L}\varphi)(t,x)=\frac{1}{2}\mathrm{tr}\left(a(x)D^2\varphi(t,x)\right)+\langle b(x),\nabla \varphi(t,x)\rangle,
\end{align*} 
for a sufficiently smooth function $\varphi:\R^{d+1}_{0,T}\to\R$.

\section{Properties of the approximating problems and stability estimates}\label{sec:approx}
In this section we study the game described in Section \ref{sec:penprobg}. Our first lemma shows that we can restrict the class of admissible controls to those with bounded expectation uniformly in $x$ in compact sets. In the proof we use $\inf_{t\in[0,T]}f(t) = f(T) >0$ from Assumption \ref{ass:gen2}. Recall that we formally set $\overline u^0=\overline v$ and $\underline u^0=\underline v$.
\begin{lemma}\label{lem:3.3}
There is $K_2>0$ such that for any $(t,x)\in\R^{d+1}_{0,T}$ and $\gamma\in[0,1]$
\begin{align*}
\overline u^\gamma(t,x)=\!\inf_{(n,\nu)\in\cA_{t,x}^{d,opt}}\sup_{\tau\in\cT_t}\cJ^\gamma_{t,x}(n,\nu,\tau),\quad
\underline u^\gamma(t,x)= \sup_{\tau\in\cT_t}\inf_{(n,\nu)\in\cA_{t,x}^{d,opt}}\cJ^\gamma_{t,x}(n,\nu,\tau),
\end{align*}
where $\cA_{t,x}^{d,opt} \coloneqq \big\{(n,\nu)\in\cA^d\,\big|\, \E_x[\nu_{T-t}]\leq K_2(1+|x|_d)\big\}$. (When $\gamma=0$ we must use $\cA_{t,x}^{d_0,opt}\coloneqq\cA_{t,x}^{d,opt}\cap\cA^{d_0}$ instead of $\cA_{t,x}^{d,opt}$.)
\end{lemma}
\begin{proof}
Let $(e_1,0)\in\cA^d$ be the null control, where $e_1=(1,0,\ldots, 0)\in\R^d$, and denote $X=X^{[e_1,0]}$ (notice that $X^{[e_1,0]}=X^{[e_1,0],\gamma}$ for all $\gamma\in[0,1]$). We have
\begin{align}\label{eq:estpayoffnu_1}
\begin{aligned}
\overline u^\gamma(t,x)\leq&\, \sup_{\tau\in\cT_t}\cJ^{\gamma}_{t,x}(e_1,0,\tau)
=\sup_{\tau\in\cT_t}\E_x\Big[e^{-r\tau}g(t\!+\!\tau,X_\tau)\!+\!\int_0^\tau\!\! e^{-rs}h(t\!+\!s,X_s)\ud s\Big]\\
\leq&\, K_1(1+T)\Big(1+\E_x\Big[\sup_{s\in[0,T]}|X_s|^2_d\Big]^{1/2}\Big)\leq C_1(1+|x|_d),
\end{aligned}
\end{align}
where the second inequality uses the sub-linear growth of $g$ and $h$, the third inequality is by standard estimates for SDEs with linearly growing coefficients (\cite[Cor.\ 2.5.10]{krylov1980controlled}). The constant $C_1>0$ depends only on $T$, $D_2$ and $K_1$ from \eqref{eq:lingrowcff21} and Assumption \ref{ass:gen2}(ii), respectively. Since $0 \le \overline u^\gamma$,  by \eqref{eq:estpayoffnu_1} we can restrict admissible controls in $\overline u^\gamma$ to the class $\cA^{d,sub}_{t,x} \coloneqq \{(n,\nu)\in\cA^d\,|\, \sup_{\tau\in\cT_t}\cJ^\gamma_{t,x}(n,\nu,\tau)\leq C_1(1+|x|_d)\}$.

A similar argument applies for the lower value $\underline u^\gamma$. As in \eqref{eq:estpayoffnu_1}, 
\[
\inf_{(n,\nu)\in\cA^{d}}\cJ^\gamma_{t,x}(n,\nu,\tau)
= \inf_{(n,\nu)\in\cA_{t,x}^{d,sub}}\cJ^\gamma_{t,x}(n,\nu,\tau),
\]
for any $(t, x)$ and $\tau$. So we can also restrict controls to $\cA_{t,x}^{d,sub}$ in the definition of $\underline{u}^\gamma$.

It remains to show that $\cA^{d, sub}_{t,x} \subseteq \cA^{d, opt}_{t,x}$. To this end, recall that $f>0$.
For $(n,\nu)\in\cA^{d,sub}_{t,x}$ we have
\begin{align}\label{eq:stmoptcont_1}
\begin{aligned}
&\E_x\big[\nu_{T-t}\big]\leq e^{r(T-t)}\Big(\min_{s\in[0,T-t]}f(t+s)\Big)^{-1}\E_x\Big[\int_{[0,T-t]}\!e^{-rs}f(t+s)\,\ud \nu_{s}\Big]\\
& \le \frac{e^{r(T-t)}}{f(T)} \E_x\Big[\int_{[0,T-t]}\!e^{-rs}f^\gamma(t+s)\,\ud \nu_{s}\Big] 
\leq \frac{e^{rT}}{f(T)}\cJ^\gamma_{t,x}(n,\nu,T-t),
\end{aligned}
\end{align}
where the second inequality uses that $f\le f^\gamma$ and $f$ is non-increasing in time and the third inequality follows from $g,h\ge 0$. Using \eqref{eq:estpayoffnu_1} and \eqref{eq:stmoptcont_1}, we have $\E_x[\nu_{T-t}]\le \frac{e^{rT}}{f(T)}\overline u^\gamma(t,x)\le K_2(1+|x|_d)$, with $K_2=\frac{e^{rT}C_1}{f(T)}$.
This concludes the proof because $\cA^{d, sub}_{t,x} \subseteq \cA^{d, opt}_{t,x} \subseteq \cA^{d}$ and in the first part of the proof we have shown that $\cA^{d}$ can be replaced by $\cA^{d, sub}_{t,x}$ in the definitions of $\overline{u}^\gamma$ and $\underline{u}^\gamma$.
\end{proof}

From now on we assume stronger conditions than in Assumption \ref{ass:gen2} for the sake of simplicity of exposition. These will be relaxed in Section \ref{sec:relax41}. In particular, throughout this section we enforce
\begin{assumption}\label{ass:gen3}
Functions $f:[0,T]\to(0,\infty)$, $g,h:\R^{d+1}_{0,T}\to[0,\infty)$ satisfy:
\begin{itemize} 
\item[(i)] $g \in C^{\infty}_{c,\,\rm sp}(\R^{d+1}_{0,T})$ and $h\in C^{\infty}_{c,\,\rm sp}(\R^{d+1}_{0,T})$; 
\item[(ii)] $f\in C^\infty([0,T])$, non-increasing and strictly positive;
\item[(iii)] For all $(t,x)\in\R^{d+1}_{0,T}$, it holds
\begin{align*}
|\nabla^{0} g(t,x)|_{d_0}\leq f(t).
\end{align*}
\end{itemize}
\end{assumption}
We notice that the assumptions of infinite continuous differentiability and compact support imply the existence of a constant $K\in(0,\infty)$ such that:
\begin{itemize}
\item[(iv)] $f$, $g$ and $h$ are bounded and, for all $0\leq s< t\leq T$ and all $x,y\in\R^{d}$,
\begin{align}\label{eq:ghlipL_1}
|g(t,x)-g(s,y)|+|h(t,x)-h(s,y)|\le K\big(|x-y|_d+(t-s)\big);
\end{align}
\item[(v)] For all $(t,x)\in\R^{d+1}_{0,T}$, it holds
\begin{align*}
(h+\partial_tg+\mathcal{L}g-rg)(t,x)\ge -K.
\end{align*}
\end{itemize}

In the construction of $f^\gamma$ in \eqref{eq:defnfgamma} we take $K>0$ as in (iv) and (v) above. Then, Assumption \ref{ass:gen3}(iii) and \eqref{eq:ghlipL_1} imply for all $(t,x)\in\R^{d+1}_{0,T}$ 
\begin{align*}
|\nabla g(t,x)|_{\gamma}^2=&\,|\nabla^0 g(t,x)|_{d_0}^2+\gamma |\nabla^1 g(t,x)|_{d_1}^2\leq f^2(t)+\gamma K^2=(f^\gamma(t))^2.
\end{align*}
For the game in Section \ref{sec:penprobg} (with expected payoff $\cJ^\gamma_{t,x}$) our Assumption \ref{ass:gen3} implies Assumption 3.2 in \cite{bovo2022variational}, and our Assumption \ref{ass:gen1} implies Assumption 3.1 in \cite{bovo2022variational}. The variational inequality that identifies the value of such game is the following:
\begin{align}\label{eq:VI}
\begin{split}
&\min\big\{\max\big\{\partial_t u+\cL u-ru+h,g-u\big\},f^\gamma-|\nabla u|_\gamma\big\}=0,\quad\text{a.e.\ in $\R^{d+1}_{0,T}$},\\
&\max\big\{\min\big\{\partial_t u+\cL u-ru+h,f^\gamma-|\nabla u|_\gamma\big\},g-u\big\}=0,\quad\text{a.e.\ in $\R^{d+1}_{0,T}$},
\end{split}
\end{align}
with terminal condition $u(T,x)=g(T,x)$ and growth condition $|u(t,x)|\le c(1+|x|_d)$, for a suitable $c>0$. A simple adaptation of the results from \cite{bovo2022variational} leads to the next theorem. Details of the changes to the original proof of \cite[Thm.\ 3.3]{bovo2022variational} are given in Appendix for completeness.
\begin{theorem}\label{thm:usolvar_ga}
The game described above admits a value (i.e., $\underline u^\gamma=\overline u^\gamma$) and the value function $u^\gamma$ is the maximal solution\footnote{Maximal means that if $w\in W^{1,2,p}_{\ell oc}(\R^{d+1}_{0,T})$, for all $p\in[1,\infty)$, is another solution of \eqref{eq:VI}, then $u^\gamma(t,x)\ge w(t,x)$ for all $(t,x)\in\R^{d+1}_{0,T}$.} of \eqref{eq:VI} in the class $W^{1,2,p}_{\ell oc}(\R^{d+1}_{0,T})$ for all $p\in[1,\infty)$. Moreover, for any given $(t,x)\in\R^{d+1}_{0,T}$ and any admissible control $(n,\nu)\in\cA^{d}$, the stopping time defined $\P_x$-a.s.\ as
\begin{align}\label{eq:taustar_ga}
\tau_*^\gamma\coloneqq\inf\big\{s\geq0\,|\,u^\gamma(t+s,X_s^{[n,\nu],\gamma})=g(t+s,X_s^{[n,\nu],\gamma})\big\}
\end{align}
is optimal for the stopper.
\end{theorem}

\begin{remark}\label{rem:bddug}
Thanks to the boundedness and positivity of $f,g,h$, the value function of the game $u^\gamma$ is bounded and non-negative. The upper bound is obtained by taking the sub-optimal control $(n,\nu)\equiv(e_1,0)$ with $e_1=(1,0,\ldots, 0)$. In turn, by the maximality of $u^\gamma$ across the solutions of \eqref{eq:VI}, we have that any solution of \eqref{eq:VI} in $W^{1,2,p}_{\ell oc}(\R^{d+1}_{0,T})$ is bounded. 
\end{remark}

The family of stopping times $(\tau^\gamma_*)_{\gamma>0}$ is optimal for the stopper in the corresponding family of games with values $(u^\gamma)_{\gamma>0}$. However, it turns out that studying the convergence of $\tau^\gamma_*$ for $\gamma\downarrow 0$ is not an easy task. For that reason we introduce another family of stopping times and we prove some of its useful properties. 
For $\gamma>0$ and $(n,\nu)\in\cA^d$, let 
\begin{align*}
&\sigma_*^\gamma\coloneqq\inf\{s\geq 0| u^\gamma(t+s,X_{s-}^{[n,\nu],\gamma})-g(t+s,X_{s-}^{[n,\nu],\gamma})=0\},
\end{align*}
and define
\begin{align}\label{eq:thetagamma}
\theta^{\gamma}_*\coloneqq\tau^\gamma_*\wedge\sigma^\gamma_*.
\end{align} 
Notice that given $(t,x)\in\R^{d+1}_{0,T}$ and $(n,\nu)\in\cA^d$ the stopping time depends on both $(t,x)$ and $(n,\nu)$ via the controlled dynamics $X^{[n,\nu],\gamma}$ (Remark \ref{rem:opt}). Therefore we sometimes use the notation $\theta^\gamma_* = \theta^\gamma_* (t,x; n, \nu)$ or the shorter $\theta^\gamma_*=\theta^\gamma_*(n,\nu)$.

Next we show that $\theta^{\gamma}_*$ is optimal for the stopper in the game with value $u^\gamma$.
\begin{lemma}\label{lem:convst}
Fix $(t,x)\in\R^{d+1}_{0,T}$. We have 
\begin{equation}\label{eqn:theta_gamma_opt}
u^\gamma(t,x)\le \cJ^\gamma_{t,x}(n,\nu,\theta^{\gamma}_*),
\end{equation} 
for any $(n,\nu)\in\cA^d$. Furthermore,
\[
u^\gamma(t,x)=\inf_{(n,\nu)\in\cA^d}\cJ^\gamma_{t,x}(n,\nu,\theta^{\gamma}_*(n,\nu)),
\] 
hence $\theta^{\gamma}_*$ is optimal for the stopper in the game with value $u^\gamma$.
\end{lemma}
\begin{proof}
With no loss of generality we assume
\[
\cC_\gamma=\{(t,x)\in\R^{d+1}_{0,T}:u^\gamma(t,x)> g(t,x)\}\neq\varnothing.
\]
If $\cC_\gamma=\varnothing$ then $\theta^\gamma_*=0$ and the lemma trivially holds. 
Next we adapt an argument from the verification result for singular control, \cite[Thm. VIII.4.1]{fleming2006controlled}, to overcome the lack of smoothness of $u^\gamma$.

Let $(\zeta_k)_{k\in\N}$ be a standard family of mollifiers and consider the sequence $(w_k^\gamma)_{k\in\N}\subset C^{\infty}(\R^{d+1}_{0,T})$, obtained by convolution $w_k^\gamma\coloneqq u^\gamma*\zeta_k$. Since $u^\gamma$ belongs to $W^{1,2,p}_{\ell oc}(\R^{d+1}_{0,T})$ which is compactly embedded in $C^{0,1,\alpha}_{\ell oc}(\R^{d+1}_{0,T})$ for $p > d+2$ and some $\alpha \in (0,1)$, we have $w_k^\gamma\to u^\gamma$ and $\nabla w_k^\gamma\to \nabla u^\gamma$ uniformly on compact sets, as $k\to\infty$; moreover, $\partial_t w_k^\gamma\to\partial_t u^\gamma$ and $D^2 w_k^\gamma\to D^2 u^\gamma$ strongly in $L^p_{\ell oc}(\R^{d+1}_{0,T})$ for all $p\in[1,\infty)$, as $k\to\infty$ (see, e.g., arguments in Thm.~5.3.1 and Appendix~C.4 in \cite{evans10}).  For notational simplicity, denote the operator $(\partial_t+\cL-r)$ by $\tilde{\cL}$. 

Standard calculations based on integration by parts yield
\[
\partial_t w_k^\gamma= (\partial_t u^\gamma)*\zeta_k,\quad \partial_{x_j} w_k^\gamma= (\partial_{x_j} u^\gamma)*\zeta_k,\quad \partial_{x_ix_j} w_k^\gamma =(\partial_{x_ix_j} u^\gamma)*\zeta_k,
\]
and therefore
\begin{align*}
&|\,(\tilde{\cL}u^\gamma*\zeta_k)(t,x)-(\tilde{\cL}w_k^\gamma)(t,x)|\\
&=\Big|\int_{\R^{d+1}_{0,T}}\Big(\sum_{i,j=1}^d(a_{ij}(y)-a_{ij}(x))\partial_{x_ix_j}u^\gamma(s,y)+\sum_{i=1}^d(b_{i}(y)-b_{i}(x))\partial_{x_i}u^\gamma(s,y)\Big)\zeta_k(t-s,x-y)\,\ud s\ud y\Big|.
\end{align*}
Since first and second order derivatives of $u^\gamma$ belong to $L^p_{\ell oc}(\R^{d+1}_{0,T})$ for any $p\in[1,\infty)$, then H\"older's inequality and continuity of $a$ and $b$ yield for any compact $\Sigma\subset\R^{d+1}_{0,T}$
\begin{align}\label{eq:convwklwk}
\lim_{k\to\infty}Q_k^\Sigma =0,
\end{align}
where $Q_k^\Sigma:= \sup_{(t,x)\in \Sigma}|(\tilde{\cL}u^\gamma*\zeta_k)(t,x)-(\tilde{\cL}w_k^\gamma)(t,x)|$. Since $u^\gamma$ is a solution of \eqref{eq:VI}, we have $(\tilde{\cL} u^\gamma+h)(t,x)\geq 0$ for almost every $(t,x)\in\cC_{\gamma}$ and therefore
\begin{align}\label{eq:lwk<hk}
\chi^\gamma_k(t,x)\coloneqq\big((\tilde{\cL} u^\gamma+h)*\zeta_k\big)(t,x)\geq 0,
\end{align}
for all $(t,x)\in\cC_{\gamma}$. Finally, denoting $h_k = h * \zeta_k$ and $M^\Sigma_k:=\sup_{(t,x)\in\Sigma}\big|h_k(t,x)-h(t,x)\big|$ we have
\begin{align}\label{eq:hhk}
\lim_{k\to\infty}M^\Sigma_k=0.
\end{align}
From \eqref{eq:convwklwk}, \eqref{eq:lwk<hk} and \eqref{eq:hhk} we have
\begin{align}\label{eq:limsup-w}
\begin{split}
&\liminf_{k\to\infty}\inf_{(t,x)\in \Sigma\cap\,\cC_\gamma}\Big((\tilde{\cL}w_k^\gamma)(t,x)+h(t,x)\Big)\\
&\ge \liminf_{k\to\infty}\Big(\inf_{(t,x)\in \Sigma\cap\,\cC_\gamma}\chi^\gamma_k(t,x)-Q^\Sigma_k-M^\Sigma_k\Big) \ge 0.
\end{split}
\end{align}

Fix $(n,\nu)\in\cA^d$ and let $\rho_m=\inf\{s\ge 0 | X^{[n,\nu],\gamma}_s\notin B_m\}\wedge(T-t)$. By an application of Dynkin's formula we obtain
\begin{align*}
w_k^\gamma(t,x)= \E_x\Big[&e^{-r(\theta^\gamma_*\wedge\rho_{m})}w_k^\gamma\big(t\!+\!\theta_*^\gamma\wedge\rho_{m},X_{\theta_*^\gamma\wedge\rho_{m}}^{[n,\nu],\gamma}\big)\!-\!\int_{0}^{\theta_*^\gamma\wedge\rho_{m}}\!\!\!e^{-rs}\tilde\cL w_k^\gamma(t\!+\!s,X_s^{[n,\nu],\gamma})\ud s \\
&-\!\int_{0}^{\theta_*^\gamma\wedge\rho_{m}}\!\!e^{-rs}\langle\nabla w_k^\gamma(t\!+\!s,X_{s-}^{[n,\nu],\gamma}), n_s\rangle_\gamma\,\ud \nu_s^c\\
&-\! \sum_{s\leq\theta_*^\gamma\wedge\rho_{m}}\!\!\!e^{-rs}\!\!\int_{0}^{\Delta\nu_{s}}\!\!\langle\nabla w_k^\gamma\big(t\!+\!s,X_{s-}^{[n,\nu],\gamma}\!+\!\lambda n_s\big),n_s\rangle_\gamma\, \ud\lambda \Big].
\end{align*}
The contribution to $w_k^\gamma$ of the final jump of $X^{[n,\nu],\gamma}$ at $\theta_*^\gamma\wedge\rho_{m}$ is removed using
\begin{multline}
\label{eqn:deal_jump}
w_k^\gamma\big(t+\theta_*^\gamma\wedge\rho_{m},X_{\theta_*^\gamma\wedge\rho_{m}}^{[n,\nu],\gamma}\big) \\
=
w_k^\gamma\big(t+\theta_*^\gamma\wedge\rho_{m},X_{\theta_*^\gamma\wedge\rho_{m}-}^{[n,\nu],\gamma}\big)
+ \int_{0}^{\Delta\nu_{\theta_*^\gamma\wedge\rho_{m}}}\!\!\langle\nabla w_k^\gamma(t+s,X_{\theta_*^\gamma\wedge\rho_{m}-}^{[n,\nu],\gamma}+\lambda n_{\theta_*^\gamma\wedge\rho_{m}}),n_{\theta_*^\gamma\wedge\rho_{m}}\rangle_\gamma\, \ud\lambda.
\end{multline}
Then, substituting the latter in the expectation yields
\begin{align*}
w_k^\gamma(t,x)= \E_x\Big[&e^{-r(\theta^\gamma_*\wedge\rho_{m})}w_k^\gamma\big(t+\theta_*^\gamma\wedge\rho_{m},X_{\theta_*^\gamma\wedge\rho_{m}-}^{[n,\nu],\gamma}\big)\!-\!\int_{0}^{\theta_*^\gamma\wedge\rho_{m}}\!\!\!e^{-rs}\tilde\cL w_k^\gamma(t+s,X_s^{[n,\nu],\gamma})\ud s \\
&-\int_{0}^{\theta_*^\gamma\wedge\rho_{m}}\!\!e^{-rs}\langle\nabla w_k^\gamma(t+s,X_{s-}^{[n,\nu],\gamma}), n_s\rangle_\gamma\,\ud \nu_s^c\\
&- \sum_{s<\theta_*^\gamma\wedge\rho_{m}}\!\!\!e^{-rs}\!\!\int_{0}^{\Delta\nu_{s}}\!\!\langle\nabla w_k^\gamma(t+s,X_{s-}^{[n,\nu],\gamma}+\lambda n_s),n_s\rangle_\gamma\, \ud\lambda \Big].
\end{align*}
We expand $\tilde\cL w_k^\gamma(t+s,X_s^{[n,\nu],\gamma})$ as $\big(\tilde\cL w_k^\gamma + h\big)(t+s,X_s^{[n,\nu],\gamma}) - h(t+s,X_s^{[n,\nu],\gamma})$ and let $k\to\infty$. We apply the inequality \eqref{eq:limsup-w}) to the term $(\tilde\cL w_k^\gamma + h)$ and the dominated convergence theorem for the remaining terms, justified by the uniform convergence of $(w_k^\gamma,\nabla w_k^\gamma)$ to $(u^\gamma,\nabla u^\gamma)$ on compacts:
\begin{align*}
u^\gamma(t,x)\le \E_x\Big[&e^{-r(\theta^\gamma_*\wedge\rho_{m})}u^\gamma(t+\theta_*^\gamma\wedge\rho_{m},X_{\theta_*^\gamma\wedge\rho_{m}-}^{[n,\nu],\gamma})\!+\!\int_{0}^{\theta_*^\gamma\wedge\rho_{m}}\!\!\!e^{-rs}h(t+s,X_s^{[n,\nu],\gamma})\ud s \\
&-\!\int_{0}^{\theta_*^\gamma\wedge\rho_{m}}\!\!e^{-rs}\langle\nabla u^\gamma(t\!+\!s,X_{s-}^{[n,\nu],\gamma}), n_s\rangle_\gamma\,\ud \nu_s^c\\
&- \sum_{s<\theta_*^\gamma\wedge\rho_{m}}\!\!\!e^{-rs}\!\!\int_{0}^{\Delta\nu_{s}}\!\!\langle\nabla u^\gamma(t+s,X_{s-}^{[n,\nu],\gamma}+\lambda n_s),n_s\rangle_\gamma\, \ud\lambda \Big].
\end{align*} 
Notice that $\P_x(\rho_m<\theta^\gamma_*)\downarrow 0$ as $m\to\infty$. Then, in the limit as $m\to\infty$ the dominated convergence theorem yields (recall $u^\gamma$ and $h$ are bounded, $|\nabla u^\gamma|_\gamma\le f^\gamma$ and $\E_x [\nu_{T-t}] < \infty$)
\begin{align}\label{eq:ev0}
u^\gamma(t,x)\le \E_x\Big[&e^{-r\theta^\gamma_*}u^\gamma(t+\theta_*^\gamma,X_{\theta_*^\gamma-}^{[n,\nu],\gamma})\!+\!\int_{0}^{\theta_*^\gamma}\!\!\!e^{-rs}h(t+s,X_s^{[n,\nu],\gamma})\ud s \notag\\
&-\int_{0}^{\theta_*^\gamma}\!\!e^{-rs}\langle\nabla u^\gamma(t\!+\!s,X_{s-}^{[n,\nu],\gamma}), n_s\rangle_\gamma\,\ud \nu_s^c\\
&\!-\! \sum_{s<\theta_*^\gamma}\!\!e^{-rs}\!\!\int_{0}^{\Delta\nu_{s}}\!\!\langle\nabla u^\gamma(t\!+\!s,X_{s-}^{[n,\nu],\gamma}\!+\!\lambda n_s),n_s\rangle_\gamma\, \ud\lambda \Big].\notag
\end{align}
On $\{\tau^\gamma_*\le \sigma^\gamma_*\}$ we have 
\begin{equation}\label{eqn:tau_less_sigma}
\begin{aligned}
u^\gamma(t+\theta_*^\gamma,X_{\theta_*^\gamma-}^{[n,\nu],\gamma})
&=u^\gamma(t+\tau_*^\gamma,X_{\tau_*^\gamma-}^{[n,\nu],\gamma})\\
&=u^\gamma(t+\tau_*^\gamma,X_{\tau_*^\gamma}^{[n,\nu],\gamma}) - \int_{0}^{\Delta\nu_{\tau_*^\gamma}}\!\!\langle\nabla u^\gamma(t\!+\!\tau_*^\gamma,X_{\tau_*^\gamma-}^{[n,\nu],\gamma}\!+\!\lambda n_{\tau_*^\gamma}),n_{\tau_*^\gamma}\rangle_\gamma\, \ud\lambda\\
&=g(t+\tau_*^\gamma,X_{\tau_*^\gamma}^{[n,\nu],\gamma}) - \int_{0}^{\Delta\nu_{\tau_*^\gamma}}\!\!\langle\nabla u^\gamma(t\!+\!\tau_*^\gamma,X_{\tau_*^\gamma-}^{[n,\nu],\gamma}\!+\!\lambda n_{\tau_*^\gamma}),n_{\tau_*^\gamma}\rangle_\gamma\, \ud\lambda\\
&\le g(t+\tau_*^\gamma,X_{\tau_*^\gamma}^{[n,\nu],\gamma}) + f^\gamma(t + \tau_*^\gamma) \Delta\nu_{\tau_*^\gamma},
\end{aligned}
\end{equation}
where the third equality is by the definition of $\tau^\gamma_*$, the continuity of $u^\gamma$ and $g$, and the right-continuity of $t\mapsto X^{[n,\nu],\gamma}_t$; the inequality follows from $|\nabla u^\gamma|_\gamma\le f^\gamma$. We insert the estimate \eqref{eqn:tau_less_sigma} into the expression under the expectation in \eqref{eq:ev0} and apply the bound $|\nabla u^\gamma|_\gamma\le f^\gamma$ again to obtain 
\begin{equation}\label{eq:ev1}
\begin{aligned}
&e^{-r\theta^\gamma_*}u^\gamma(t+\theta_*^\gamma,X_{\theta_*^\gamma-}^{[n,\nu],\gamma}) +\!\int_{0}^{\theta_*^\gamma}\!\!\!e^{-rs}h(t+s,X_s^{[n,\nu],\gamma})\ud s\\
&-\int_{0}^{\theta_*^\gamma}\!\!e^{-rs}\langle\nabla u^\gamma(t\!+\!s,X_{s-}^{[n,\nu],\gamma}), n_s\rangle_\gamma\,\ud \nu_s^c
- \sum_{s<\theta_*^\gamma}\!\!e^{-rs}\!\!\int_{0}^{\Delta\nu_{s}}\!\!\langle\nabla u^\gamma(t\!+\!s,X_{s-}^{[n,\nu],\gamma}\!+\!\lambda n_s),n_s\rangle_\gamma\, \ud\lambda \\
&\!\le e^{-r\theta^\gamma_*}g(t+\theta_*^\gamma,X_{\theta^\gamma_*}^{[n,\nu],\gamma}) + \int_{0}^{\theta_*^\gamma}\!\!\!e^{-rs}h(t\!+\!s,X_s^{[n,\nu],\gamma})\ud s + \int_{[0,\theta^\gamma_*]}\!\!e^{-rs}f^\gamma(t\!+\!s)\ud\nu_s.
\end{aligned}
\end{equation}

On the event $\{\sigma^\gamma_*<\tau^\gamma_*\}$, the arguments are more involved. We start from showing that 
\[
u^\gamma(t+\sigma_*^\gamma,X_{\sigma_*^\gamma-}^{[n,\nu],\gamma})=g(t+\sigma_*^\gamma,X_{\sigma_*^\gamma-}^{[n,\nu],\gamma}).
\] 
Since $\sigma^\gamma_* <\tau^\gamma_* \le T-t$, we have $(u^\gamma-g)(t+\sigma_*^\gamma,X_{\sigma_*^\gamma}^{[n,\nu],\gamma}) > 0$. The process 
\[
s \mapsto (u^\gamma - g)\big(t+s, X_{s}^{[n,\nu],\gamma}\big)
\] 
is right-continuous due to continuity of $u^\gamma$ and $g$ and right-continuity of $t\mapsto X_{t}^{[n,\nu],\gamma}$. Using this fact we deduce that for $\P_x$-almost every $\omega$ there is $\eps(\omega), \delta(\omega) > 0$ such that 
\[
(u^\gamma - g)\big(t+s, X_{s}^{[n,\nu],\gamma}\big) > \eps(\omega), \qquad \forall\, s \in [\sigma^\gamma_*, \sigma^\gamma_* + \delta(\omega)].
\]
This means that $(\sigma^\gamma_*, \sigma^\gamma_* + \delta(\omega)] \cap \big\{ s \ge 0\,\big|\, (u^\gamma - g)\big(t+s, X_{s-}^{[n,\nu],\gamma}\big) = 0 \big\} = \varnothing$. Hence, by the definition of $\sigma_*^\gamma$, we conclude that $(u^\gamma - g)\big(t+\sigma_*^\gamma, X_{\sigma_*^\gamma-}^{[n,\nu],\gamma}\big) = 0$. 

We now rewrite
\begin{equation}\label{eqn:sigma_less_tau}
\begin{aligned}
u^\gamma(t+\theta_*^\gamma,X_{\theta_*^\gamma-}^{[n,\nu],\gamma})&= u^\gamma(t+\sigma_*^\gamma,X_{\sigma_*^\gamma-}^{[n,\nu],\gamma})
= g(t+\sigma_*^\gamma,X_{\sigma_*^\gamma-}^{[n,\nu],\gamma})\\
&= g(t+\sigma_*^\gamma,X_{\sigma_*^\gamma}^{[n,\nu],\gamma})
- \int_{0}^{\Delta\nu_{\sigma_*^\gamma}}\!\!\langle\nabla g(t\!+\!\sigma_*^\gamma,X_{\sigma_*^\gamma-}^{[n,\nu],\gamma}\!+\!\lambda n_{\sigma_*^\gamma}),n_{\sigma_*^\gamma}\rangle_\gamma\, \ud\lambda \\
&\le g(t+\sigma_*^\gamma,X_{\sigma_*^\gamma}^{[n,\nu],\gamma}) + f^\gamma(t + \sigma_*^\gamma) \Delta\nu_{\sigma_*^\gamma},
\end{aligned}
\end{equation}
where in the first line we use the identity $u^\gamma(t+\sigma_*^\gamma,X_{\sigma_*^\gamma-}^{[n,\nu],\gamma})=g(t+\sigma_*^\gamma,X_{\sigma_*^\gamma-}^{[n,\nu],\gamma})$ proved above and in the last line the bound $|\nabla g|_\gamma\le f^\gamma$. We insert the estimate \eqref{eqn:sigma_less_tau} into the expression under the expectation on the right-hand side of \eqref{eq:ev0} and apply the bound $|\nabla u^\gamma|_\gamma\le f^\gamma$ to obtain \eqref{eq:ev1}.

Substituting \eqref{eq:ev1} inside the expectation on the right-hand side of \eqref{eq:ev0} yields
\begin{align*}
u^\gamma(t,x)&\le\E_x\Big[e^{-r\theta_*^\gamma}g(t\!+\!\theta_*^\gamma,X_{\theta_*^\gamma}^{[n,\nu],\gamma})+\int_{0}^{\theta_*^\gamma}\!\!e^{-rs}h(t\!+\!s,X_s^{[n,\nu],\gamma})\ud s\!+\!\int_{[0,\theta_*^\gamma]}\!\!e^{-rs} f^\gamma(t\!+\!s)\,\ud \nu_s \Big]\notag\\
&=\cJ_{t,x}^\gamma(n,\nu,\theta^\gamma_*),
\end{align*}
which proves the first claim in the lemma. By arbitrariness of the pair $(n,\nu)\in\cA^d$ we conclude
\[
u^\gamma(t,x)\le\inf_{(n,\nu)\in\cA^d} \cJ_{t,x}^\gamma(n,\nu,\theta^\gamma_*)\le u^\gamma(t,x),
\]
hence proving the second statement of the lemma. 
\end{proof}
\begin{remark}
When the sets $\{u^\gamma=g\}$ and $\{|\nabla u^\gamma|_\gamma=f\}$ are disjoint, heuristic arguments based on classical verification theorems suggest that the controller and the stopper do not act simultaneously. In particular, this means that with no loss of generality we should be able to restrict the class of admissible pairs $(n,\nu)$ to those for which $\Delta\nu_{\theta^\gamma_*}=0$ so that $\tau^\gamma_*(n,\nu)=\sigma^\gamma_*(n,\nu)$. This type of analysis is left for future work on more concrete examples. 
\end{remark}

\subsection{Some stability estimates}\label{sec:stability}
We next provide a stability estimate in $L^1$ for the approximating process. The proof uses a generalisation of \cite[Lem.\ 5.1]{deangelis2019numerical} which is given as Lemma \ref{lem:DGI_jumps} in Appendix for completeness.

\begin{proposition}\label{thm:convgamma0}
Fix $(t,x)\in\R^{d+1}_{0,T}$ and a treble $[(n,\nu),\tau]\in\cA^d\times\cT_t$. Then, there exists $(\bar{n},\bar{\nu})\in \cA^{d_0}$ such that 
\begin{align}
\E_x\Big[\big|X_\tau^{[n,\nu],\gamma}-X_\tau^{[\bar{n},\bar{\nu}]}\big|_d\Big]\leq \gamma K_3\E_x[\nu_{T-t}],
\end{align}
where $K_3>0$ is a constant depending only on $d$, $D_1$ and $T$. 
\end{proposition}
\begin{proof}
For each pair $(n,\nu)\in\cA^{d}$, setting $n_s=n(s)=(n_{[d_0]}(s),n_{[d_1]}(s))\in\R^{d_0}\times\R^{d_1}$, we can define a pair $(\bar{n},\bar{\nu})\in \cA^{d_0}$ as follows: for $i=1,\ldots, d_0$, we set 
\begin{equation}\label{eq:barnuhbard_1}
\begin{split}
&\bar{n}^i_{s}=\begin{cases}&\frac{n^i_{s}}{|n_{[d_0]}(s)|_{d_0}},\qquad \text{if}\ |n_{[d_0]}(s)|_{d_0}\neq0,\\ 
&(1,0,\ldots, 0),\quad \text{if}\ |n_{{[d_0]}}(s)|_{d_0}=0;
\end{cases}\\
&\bar{\nu}_s=\int_0^s\!|n_{[d_0]}(r)|_{d_0}\,\ud\nu_r,
\end{split}
\end{equation}
and $\bar{n}^i_s = 0$, $i=d_0+1, \ldots, d$. By construction the process $(\bar{n}_s)_{s\in[0,T]}\in\R^{d}$ is progressively measurable, hence $\bar\nu$ is adapted, right-continuous and non-decreasing with 
\begin{align}\label{eq:nubar}
\int_0^s \bar{n}^i_{r}\ud \bar{\nu}_r=\int_0^s n^i_{r}\ud \nu_r
\end{align}
for all $s\in[0,T]$ and $i=1,\ldots, d_0$.
 
For $\tau\in\cT_t$ and $(n,\nu)\in\cA^{d}$, let
\begin{align*}
\tau_R\coloneqq \inf\big\{s\geq 0\,\big|\, |X_{s}^{[n,\nu],\gamma}|_d\vee|X^{[\bar n,\bar \nu]}_{s}|_d\geq R\big\}\wedge T.
\end{align*}
Denote the stopped processes $(X^{[n,\nu],\gamma}_{s\wedge\tau\wedge\tau_R})_{s\in[0,T]}$ and $(X^{[\bar n,\bar \nu]}_{s\wedge\tau\wedge\tau_R})_{s\in[0,T]}$ by $(X^{\gamma,R}_s)_{s \in[0,T]}$ and $(X^{R}_s)_{s \in[0,T]}$, respectively. Let $J^{\,\gamma,R}\coloneqq X^{\gamma,R}-X^{R}$ and notice that $J^{\,\gamma,R}$ is a c\`adl\`ag semimartingale. To further simplify notation we set $J=J^{\,\gamma,R}$ for as long as $\gamma$ and $R$ are fixed. For each $i=1,\ldots,d$ we denote the $i$-th coordinate of $J$ by $J^{i}$. By Meyer-It\^o formula for semimartingales (see\footnote{In \cite{protter2005stochastic}, the author considers a c\`adl\`ag semi-martingale $X$ starting from $X_0=x$, whereas here we have $X_{0-}=x$. Thus, we must account for a possible jump at time zero when using \cite[Thm.\ IV.70]{protter2005stochastic}.} \cite[Thm.\ IV.70]{protter2005stochastic}), noting that $J_{0-}=0$ and that the jump part of the process $J$ is of bounded variation, we have for $s>0$ and for $i=1,\ldots, d$, 
\begin{align}\label{eq:localtime}
|J^i_{s}|=&\,\int_{[0,s\wedge\tau\wedge\tau_R]}\!\sign(J^i_{\lambda-})\,\ud J^{i,c}_{\lambda}+L_{s\wedge\tau\wedge\tau_R}^0(J^i)+\sum_{0\le \lambda\leq s\wedge\tau\wedge\tau_R}\big(|J^i_{\lambda}|-|J^i_{\lambda-}|\big)
\end{align}
where $J^{i,c}$ is the continuous part of the process $J^i$, $\sign(y)=-1$ for $y\le 0$ and $\sign(y)=1$ for $y>0$. The process $(L_t^0(J^i))_{t\ge 0}$ is the semi-martingale local time at zero of $(J^i_t)_{t\ge 0}$. 

Fix $i=1,\ldots, d_0$. Notice that $J^{i}_{\lambda}=J^{i}_{\lambda-}$ for all $\lambda\ge 0$ because of \eqref{eq:nubar}. Thus, using the form of the dynamics of $X^{\gamma,R}$ and $X^{R}$, we have
\begin{align*}
|J^{i}_{s}|=&\,\int_0^{s\wedge\tau\wedge\tau_R}\!\sign(J^{i}_{\lambda})\big(b^i(X_\lambda^{\gamma,R})-b^i(X_\lambda^{R})\big)\,\ud \lambda\\
&\,+\int_0^{s\wedge\tau\wedge\tau_R}\!\sign(J^{i}_{\lambda})\big(\kappa^i(X_{\lambda}^{\gamma,R;i})-\kappa^i(X_{\lambda}^{R;i})\big)\,\ud W_\lambda+L_{s\wedge\tau\wedge\tau_R}^0(J^i),
\end{align*}
where we notice that in the diffusion coefficient of $|J^i|$, the functions $\kappa^i$ depend only on the $i$-th coordinate $X^{\gamma,R;i}$ and $X^{R;i}$, as per (ii) in Assumption \ref{ass:gen1}. Taking expectation in the equation above and removing the martingale term ($\kappa^i$ has a linear growth so it is bounded on compacts) we get
\begin{align}\label{eq:pre_DGIlemma}
\E_x\big[|J^{i}_{s}|\big]=&\,\E_x\Big[\int_0^{s\wedge\tau\wedge\tau_R}\!\sign(J^{i}_{\lambda})\big(b^i(X_\lambda^{\gamma,R})-b^i(X_\lambda^{R})\big)\,\ud \lambda+L_{s\wedge\tau\wedge\tau_R}^0(J^{i})\Big]\notag\\
\leq &\,\E_x\Big[\int_0^{s}\!\big|b^i(X_\lambda^{\gamma,R})-b^i(X_\lambda^{R})\big|\,\ud \lambda+L_{s}^0(J^{i})\Big]\\
\leq&\E_x\Big[D_1\int_0^{s}\!\big|J_\lambda|_d\,\ud \lambda+L_{s}^0(J^{i})\Big],\notag
\end{align}
where in the first inequality we extend the integrals up to time $s$ and for the second one we use Lipschitz continuity of $b^i$ with the constant $D_1$ from Assumption \ref{ass:gen1}. In order to estimate the local time, we follow \cite[Lem.\ 5.1]{deangelis2019numerical}, which we can apply because $J^{i}$ is a continuous semimartingale: for arbitrary $\eps\in(0,1)$
\begin{align}\label{eq:DGIlemma}
\E_x\big[L_{s}^0(J^{i})\big]\leq&\, 4\eps-2\E_x\Big[\int_0^{s}\!\Big(\mathds{1}_{\{J^i_{\lambda}\in[0,\eps)\}}+\mathds{1}_{\{J^i_{\lambda}\geq \eps\}}e^{1-\frac{J^i_{\lambda}}{\eps}}\Big)\big(b^i(X_\lambda^{\gamma,R})-b^i(X_\lambda^{R})\big)\,\ud \lambda\Big] \notag\\
&\,+\frac{1}{\eps}\E_x\Big[\int_0^{s}\!\mathds{1}_{\{J^i_{\lambda}>\eps\}}e^{1-\frac{J^i_{\lambda}}{\eps}}\big(\kappa^i(X_{\lambda}^{\gamma,R;i})-\kappa^i(X_{\lambda}^{R;i})\big)^2\,\ud \lambda\Big].
\end{align}
In order to estimate the final term above we are going to use that 
\begin{align}\label{eq:lips}
\big(\kappa^i(X_{\lambda}^{\gamma,R;i})-\kappa^i(X_{\lambda}^{R;i})\big)^2 \le D_1^2 \big|J^i_\lambda\big|^2 \le 2R D_1^2\big|J^i_\lambda\big|, 
\end{align}
because $\kappa$ is Lipschitz by (i) in Assumption \ref{ass:gen1} and $|J^{i}|\le|J^{\gamma,R}|_d\le 2R$. Denote by $I_\eps$ the last expectation on the right-hand side of \eqref{eq:DGIlemma} and pick $\zeta\in(\tfrac{1}{2},1)$. We have
\begin{align*}
I_\eps=&\,\frac{1}{\eps}\E_x\Big[\int_0^{s}\!\mathds{1}_{\{J^i_{\lambda}\in(\eps,\eps^\zeta)\}}e^{1-\frac{J^i_{\lambda}}{\eps}}\big(\kappa^i(X_{\lambda}^{\gamma,R;i})-\kappa^i(X_{\lambda}^{R;i})\big)^2\,\ud \lambda\Big]\\
&\,+\frac{1}{\eps}\E_x\Big[\int_0^{s}\!\mathds{1}_{\{J^{i}_{\lambda}\geq\eps^\zeta\}}e^{1-\frac{J^{i}_{\lambda}}{\eps}}\big(\kappa^i(X_{\lambda}^{\gamma,R;i})-\kappa^i(X_{\lambda}^{R;i})\big)^2\,\ud \lambda\Big]\\
\leq&\,\frac{1}{\eps}\E_x\Big[D_1^2\int_0^{s}\! \mathds{1}_{\{J^{i}_{\lambda} \in(\eps, \eps^\zeta)\}} |J^{i}_{\lambda}|^2\,\ud \lambda+e^{1-\eps^{\zeta-1}}\int_0^{s}\!\mathds{1}_{\{J^{i}_{\lambda}\geq\eps^\zeta\}}\big(\kappa^i(X_{\lambda}^{\gamma,R;i})-\kappa^i(X_{\lambda}^{R;i})\big)^2\,\ud \lambda\Big]\\
\leq&\,D_1^2\eps^{2\zeta-1}T+\frac{2RD_1^2}{\eps}e^{1-\eps^{\zeta-1}}\E_x\Big[\int_0^{s}\!|J^{i}_{\lambda}|\,\ud \lambda\Big],
\end{align*}
where we use \eqref{eq:lips} and the bounds 
\[
e^{1-\frac{J^{i}_{\lambda}}{\eps}}\mathds{1}_{\{J^{i}_{\lambda}\in(\eps,\eps^\zeta)\}}\leq 1\quad\text{and}\quad e^{1-\frac{J^{i}_{\lambda}}{\eps}}\mathds{1}_{\{J^{i}_{\lambda}\geq \eps^\zeta\}}\leq e^{1-\eps^{\zeta-1}}.
\]
Thanks to the Lipschitz continuity of $b$, we bound the first expectation on the right-hand side of \eqref{eq:DGIlemma} by
\[
4D_1\E_x\Big[\int_0^s |J_\lambda|_d\ud \lambda\Big].
\]
Combining those upper bounds we obtain 
\begin{align}\label{eq:DGIlemma_1}
\E_x\big[L_{s}^0(J^{i})\big]\leq 4\eps\!+\!\Big(4D_1\!+\!\frac{2RD_1^2}{\eps}e^{1-\eps^{\zeta-1}}\Big)\E_x\Big[\int_0^{s}\!|J_\lambda|_d\,\ud \lambda\Big]\!+\!D_1^2\eps^{2\zeta-1}T.
\end{align}
We insert this bound into \eqref{eq:pre_DGIlemma} and obtain the following estimate:
\begin{align}\label{eq:DGIlemmahbar}
\E_x\big[|J^i_{s}|\big]\leq&\, 4\eps\!+\!\Big(5D_1\!+\!\frac{2RD_1^2}{\eps}e^{1-\eps^{\zeta-1}}\Big)\E_x\Big[\int_0^{s}\!|J_\lambda|_d\,\ud \lambda\Big]\!+\!D_1^2\eps^{2\zeta-1}T,
\end{align}
for $i=1,\ldots, d_0$. 

Coordinates $J^i$ for $i=d_0+1,\ldots, d$ are estimated slightly differently. From \eqref{eq:localtime}
\begin{align}\label{eq:DGI_jumps}
|J^{i}_{s}|=&\,\int_0^{s}\!\sign(J^i_{\lambda})\big(b^i(X_\lambda^{\gamma,R})-b^i(X_\lambda^{R})\big)\,\ud \lambda+\int_0^{s}\!\sign(J^{i}_{\lambda})\big(\kappa^i(X_{\lambda}^{\gamma,R;i})-\kappa^i(X_{\lambda}^{R;i}))\,\ud W_\lambda\notag\\
&\,+\gamma \int_{0}^{s}\! \sign(J^{i}_{\lambda-})n^i_{\lambda-} \,\ud \nu^c_{\lambda}+L_{s}^0(J^i_{\lambda})+\sum_{0\le \lambda\leq s}\big(|J^i_{\lambda}|-|J^i_{\lambda-}|\big),
\end{align}
where $\nu^c$ is the continuous part of the process $\nu$.
Notice that 
\begin{align}\label{eq:propjumpsnu1}
|J^i_{\lambda}|=&\,|J^i_{\lambda-}+\gamma n^i_{\lambda}\Delta \nu_\lambda|\leq |J^i_{\lambda-}|+\gamma\Delta \nu_\lambda,
\end{align}
which implies  
\begin{align}\label{eq:propjumpsnu2}
\gamma \int_{0}^{s}\! \sign(J^i_{\lambda})n^i_{\lambda} \,\ud \nu_{\lambda}^c+\sum_{0\le \lambda\leq s}\big(|J^i_{\lambda}|-|J^i_{\lambda-}|\big)\leq \gamma\nu_s.
\end{align}
Thus, we get from \eqref{eq:DGI_jumps} the inequality:
\begin{align*}
|J^i_{s}|\leq&\,\int_0^{s}\!\sign(J^i_{\lambda})\big(b^i(X_\lambda^{\gamma,R})-b^i(X_\lambda^{R})\big)\,\ud \lambda+\int_0^{s}\!\sign(J^i_{\lambda})\big(\kappa^i(X_{\lambda}^{\gamma,R;i})-\kappa^i(X_{\lambda}^{R;i}))\,\ud W_\lambda\\
&\,+\gamma \nu_s+L_{s}^0(J^i_{\lambda}).
\end{align*}

Since $J^i$ may have jumps, the upper bound on the local time \cite[Lemma\ 5.1]{deangelis2019numerical} does not apply. Additional terms appear as detailed in Lemma \ref{lem:DGI_jumps} in Appendix. Thus, we obtain
\begin{align}\label{eq:DGIlemmaj}
\E_x\big[L_{s}^0(J^{i})\big]\leq&\, 4\eps-2\E_x\Big[\int_0^{s}\!\Big(\mathds{1}_{\{J^i_{\lambda}\in[0,\eps)\}}+\mathds{1}_{\{J^i_{\lambda}\geq \eps\}}e^{1-\frac{J^i_{\lambda}}{\eps}}\Big)\big(b^i(X_\lambda^{\gamma,R})-b^i(X_\lambda^{R})\big)\,\ud \lambda\Big]\notag\\
&\,-2\E_x\Big[\int_0^{s}\!\Big(\mathds{1}_{\{J^i_{\lambda}\in[0,\eps)\}}+\mathds{1}_{\{J^i_{\lambda}\geq \eps\}}e^{1-\frac{J^i_{\lambda}}{\eps}}\Big)\gamma n^i_{\lambda}\,\ud \nu_\lambda^c\Big]\\
&\,+\E_x\Big[\frac{1}{\eps}\int_0^{s}\!\mathds{1}_{\{J^i_{\lambda}>\eps\}}e^{1-\frac{J^i_{\lambda}}{\eps}}\big(\kappa^i(X_{\lambda}^{\gamma,R;i})-\kappa^i(X_{\lambda}^{R;i})\big)^2\,\ud \lambda+2\gamma\sum_{0\le \lambda\leq s}\Delta\nu_{\lambda}\Big].\notag
\end{align}

Repeating the same arguments as those we used to obtain \eqref{eq:DGIlemmahbar} and additionally noticing that 
\begin{align*}
\Big|\E_x\Big[\int_0^{s}\!\Big(\mathds{1}_{\{J^i_{\lambda}\in[0,\eps)\}}+\mathds{1}_{\{J^i_{\lambda}\geq \eps\}}e^{1-\frac{J^i_{\lambda}}{\eps}}\Big)\gamma n^i_{\lambda}\,\ud \nu_\lambda^c\Big]\Big|+\E_x\Big[\gamma\sum_{0\le \lambda\leq s}\Delta\nu_{\lambda}\Big]\le \gamma \E_x[\nu_s]
\end{align*}
yields
\begin{align}\label{eq:DGIlemmahbar+1}
\E_x\big[|J^i_{s}|\big]\leq&\, 4\eps+\Big(5D_1+\frac{2RD_1^2}{\eps}e^{1-\eps^{\zeta-1}}\Big)\E_x\Big[\int_0^{s}\!|J_\lambda|_d\,\ud\lambda\Big]+D_1^2\eps^{2\zeta-1}T +3\gamma\E_x\big[\nu_{s}\big].
\end{align}
Now, combining \eqref{eq:DGIlemmahbar} and \eqref{eq:DGIlemmahbar+1} we have
\begin{align*}
\E_x\big[|J_s|_d\big]\leq&\,\sum_{i=1}^d\E_x\big[|J^i_{s}|\big]\\
\leq&\, 4d\eps+d\Big(5D_1+\frac{2RD_1^2}{\eps}e^{1-\eps^{\zeta-1}}\Big)\E_x\Big[\int_0^{s}\!|J_\lambda|_d\,\ud \lambda\Big]+d D_1^2\eps^{2\zeta-1}T +3d\gamma\E_x\big[\nu_{s}\big].
\end{align*}
Sending $\eps\downarrow0$ and recalling now that $J=J^{\gamma,R}$ and $\zeta\in(\frac12,1)$, we get
\begin{align*}
\E_x\big[|J^{\,\gamma,R}_s|_d\big]\leq&\, 5dD_1\E_x\Big[\int_0^{s}\!|J^{\,\gamma,R}_\lambda|_d\,\ud \lambda\Big] + 3d\gamma\E_x\big[\nu_{s}\big].
\end{align*}
By Gronwall's lemma, there is a constant $K_3>0$, depending only on $d$, $D_1$ and $T$, such that 
\begin{align*}
\E_x\big[|J^{\,\gamma,R}_s|_d\big]\leq& \gamma \,K_3\, \E_x\big[\nu_{T-t}\big], \quad\text{for any $s\in[0,T-t]$}.
\end{align*}
Passing to the limit as $R\to\infty$ and using Fatou's lemma, we get
\begin{align*}
\E_x\Big[|X_{s\wedge\tau}^{[n,\nu],\gamma}-X_{s\wedge\tau}^{[\bar{n},\bar{\nu}]}|_d\Big]\leq\liminf_{R\to\infty} \E_x\Big[|J^{\,\gamma,R}_s|_d\Big]\leq \gamma K_3 \E_x\big[\nu_{T-t}\big], \quad\text{for any $s\in[0,T-t]$.}
\end{align*}
Hence, the proof is completed by setting $s=T-t$ and recalling that $\tau\le T-t$.
\end{proof}
Another lemma of a similar nature allows us to compare the dynamics induced by a generic control $(n,\nu)\in\cA^{d_0}$ to its uncontrolled counterpart.
\begin{lemma}\label{lem:stability}
Fix $(t,x)\in\R^{d+1}_{0,T}$. Let $(n,\nu)\in\cA^{d_0}$ and $\tau\in\cT_t$. Then
\[
\E_x[|X^{[n,\nu]}_\tau-X^{[e_1,0]}_\tau|_d]\le K_3\E_x[\nu_{T-t}],
\]
with the same constant $K_3>0$ as in Proposition \ref{thm:convgamma0}.
\end{lemma}
\begin{proof}
Similarly as in the proof of Proposition \ref{thm:convgamma0}, we denote $X=X^{[n,\nu]}$ and $X^0=X^{[e_1,0]}$, and define
\begin{align*}
\tau_R\coloneqq \inf\big\{s\geq 0\,\big|\, |X_{s}|_d\vee|X^0_{s}|_d\geq R\big\}\wedge T.
\end{align*}
We denote the two processes $(X_{t\wedge\tau\wedge\tau_R})_{t\in[0,T]}$ and $(X^0_{t\wedge\tau\wedge\tau_R})_{t\in[0,T]}$ by $X^{R}$ and $X^{0,R}$, respectively. Let $J^{R}\coloneqq X^{R}-X^{0,R}$ and notice that $J^{R}$ is a c\`adl\`ag semimartingale. To further simplify notation we set $J=J^{R}$ for as long as $R$ is fixed. For each $i=1,\ldots,d$ we denote the $i$-th coordinate of $J$ by $J^{i}$. Now, for $i=1,\ldots, d_0$ repeating verbatim, for $\gamma = 1$, the same arguments as in the proof of \eqref{eq:DGIlemmahbar+1} we obtain
\begin{align*}
\E_x\big[|J^i_{s}|\big]\leq&\, 4\eps+\Big(5D_1+\frac{2RD_1^2}{\eps}e^{1-\eps^{\zeta-1}}\Big)\E_x\Big[\int_0^{s}\!|J_\lambda|_d\,\ud\lambda\Big]+D_1^2\eps^{2\zeta-1}T +3\E_x\big[\nu_{s}\big].
\end{align*}
Instead, for $i=d_0+1,\ldots, d$, the same arguments that yield \eqref{eq:DGIlemmahbar} now give us
\begin{align*}
\E_x\big[|J^i_{s}|\big]\leq&\, 4\eps\!+\!\Big(5D_1\!+\!\frac{2RD_1^2}{\eps}e^{1-\eps^{\zeta-1}}\Big)\E_x\Big[\int_0^{s}\!|J_\lambda|_d\,\ud \lambda\Big]\!+\!D_1^2\eps^{2\zeta-1}T.
\end{align*}
Therefore the same conclusions as in Proposition \ref{thm:convgamma0} hold but with $\gamma=1$.	
\end{proof}

Combining the above results with Lemma \ref{lem:3.3} we obtain the following corollary.

\begin{corollary}\label{lem:L1estcntrSDE}
There is a constant $K_4>0$ such that $\E_x\big[|X_\tau^{[n,\nu]}|_d\big]\leq K_4(1+|x|_d)$ for any $(t,x)\in\R^{d+1}_{0,T}$, $(n,\nu)\in\cA^{d_0,opt}_{t,x}$ and $\tau\in\cT_t$.
\end{corollary}
\begin{proof}
We observe that $\E_x\big[|X^{[n,\nu]}_\tau|_d\big]\leq \E_x\big[|X^{[n,\nu]}_\tau-X^{[e_1,0]}_\tau|_d\big]+\E_x\big[|X^{[e_1,0]}_\tau|_d\big]$. The first term is bounded using Lemma \ref{lem:stability} and \ref{lem:3.3}. Standard SDE estimates (\cite[Cor.\ 2.5.10]{krylov1980controlled}) give 
\[\E_x\big[|X^{[e_1,0]}_\tau|_d\big]\le \Big(\E_x\Big[\sup_{s\in[0,T]}|X^{[e_1,0]}_s|^2_d\Big]\Big)^{1/2}\le c(1+|x|_d),\]
for some $c>0$, thanks to Assumption \ref{ass:gen1}.
\end{proof}

\section{Convergence of the approximating problems}\label{sec:valgame}
In this section we first study the limit as $\gamma\to 0$ and then we relax the smoothness assumptions made in Assumption \ref{ass:gen3}. We observe that we could alternatively fix $\gamma$ and relax Assumption \ref{ass:gen3} before passing to the limit as $\gamma\to 0$. That approach motivates Remark \ref{rem:dd0}.

\subsection{Limits as \texorpdfstring{$\gamma\to 0$}{γ->0}}\label{sec:gammato0}
Throughout this subsection we enforce Assumptions \ref{ass:gen1} and \ref{ass:gen3}.
\begin{theorem}\label{thm:valgame_1}
The pointwise limit $u\coloneqq\lim_{\gamma\to 0}u^\gamma$ exists on $\R^{d+1}_{0,T}$. Moreover, $u$ coincides with the value of the game with payoff \eqref{eq:lowuppvfnc_1}, i.e., $u=\overline v=\underline v=v$, and there exists $C>0$ such that
\begin{align}\label{eq:convunifutheta_1}
|u^\gamma(t,x)-v(t,x)|\leq C(1+|x|_d)\gamma^{\frac{1}{2}},\quad \text{for all $(t,x)\in\R^{d+1}_{0,T}$}.
\end{align}
\end{theorem}

\begin{proof}
Let $u^\gamma$ be the value of the game described in Theorem \ref{thm:usolvar_ga}.  
We introduce $\underline{u}\coloneqq \liminf_{\gamma\to0} u^\gamma$ and $\overline{u}\coloneqq \limsup_{\gamma\to0} u^\gamma$. We want to prove that
\begin{align*}
\overline{u}(t,x)\leq \underline{v}(t,x) \quad\text{and}\quad \underline{u}(t,x)\geq \overline{v}(t,x),
\end{align*}
for all $(t,x)\in\R^{d+1}_{0,T}$, so that $\underline{u}=\overline{u}=\underline{v}=\overline{v}=v$ as claimed.

Fix $(t,x)\in\R^{d+1}_{0,T}$. We first prove that $\underline{u}\geq \overline{v}$. Let $(n,\nu)\in\cA^d$ be an $\eta$-optimal control for $u^{\gamma}(t,x)$, i.e., 
\[
\sup_{\sigma\in\cT}\cJ^\gamma_{t,x}(n,\nu,\sigma)\le u^\gamma(t,x)+\eta.
\]
With no loss of generality, thanks to Lemma \ref{lem:3.3} we can assume $(n,\nu)\in\cA^{d,opt}_{t,x}$. Consider the associated $(\bar{n},\bar{\nu})\in \cA^{d_0}$ constructed as in \eqref{eq:barnuhbard_1}. Recall the processes $X^{[n,\nu],\gamma}$ and $X^{[\bar{n},\bar{\nu}]}$ as in \eqref{eq:SDEgam} and \eqref{eq:prcXcntrll}, respectively. For notational simplicity, denote $X^\gamma = X^{[n,\nu],\gamma}$ and $X = X^{[\bar{n},\bar{\nu}]}$. Let $\tau\in\cT_t$ be an $\eta$-optimal stopping time for $\sup_{\sigma \in \cT_t} \cJ_{t,x} (\bar{n}, \bar{\nu}, \sigma)$, which implies $\overline{v}(t,x) \le \cJ_{t,x} (\bar{n}, \bar{\nu}, \tau) + \eta$. We have
\begin{align*}
&\,u^\gamma(t,x)-\overline{v}(t,x)\\
&\geq \cJ_{t,x}^\gamma(n,\nu,\tau)-\cJ_{t,x}(\bar{n},\bar{\nu},\tau)-2\eta\notag\\
&=\E_x\Big[e^{-r\tau}\Big(g(t+\tau,X_{\tau}^{\gamma})-g(t+\tau,X_{\tau})\Big)+\int_0^\tau\!e^{-rs}\big(h(t+s,X_{s}^{\gamma})-h(t+s,X_{s})\big)\,\ud s \notag\\
&\qquad+\int_{[0,\tau]}\!e^{-rs}f^\gamma(t+s)\,\ud \nu_s-\int_{[0,\tau]}\!e^{-rs}f(t+s)\,\ud \bar{\nu}_s\Big]-2\eta\notag\\
&\ge\E_x\Big[e^{-r\tau}\Big(g(t+\tau,X_{\tau}^{\gamma})-g(t+\tau,X_{\tau})\Big)+\int_0^\tau\!e^{-rs}\big(h(t+s,X_{s}^{\gamma})-h(t+s,X_{s})\big)\,\ud s\Big]-2\eta \notag\\
&\geq-K\E_x\Big[|X_{\tau}^{\gamma}-X_{\tau}|_d\Big]-K\E_x\Big[\int_0^{T-t}\!|X_{s}^{\gamma}-X_{s}|_d\,\ud s \Big]-2\eta\notag\\
&\geq -K\E_x\Big[|X_{\tau}^{\gamma}-X_{\tau}|_d\Big]-KT\sup_{s\in[0,T-t]}\E_x\Big[|X_{s}^{\gamma}-X_{s}|_d\Big] -2\eta\notag
\end{align*}
where $K>0$ is the same as in \eqref{eq:ghlipL_1}. The first inequality is by the choice of $(n,\nu)$ and $\tau$. The second inequality holds because by the definition of $\bar{\nu}$ in \eqref{eq:barnuhbard_1} we have $\tfrac{\ud \bar{\nu}_s(\omega)}{\ud\nu_s(\omega)}=|n_{[d_0]}(s,\omega)|_{d_0}\leq 1$ for all $(s,\omega)\in\R_+\times\Omega$ and $f^\gamma\geq f$ by \eqref{eq:defnfgamma}. The third inequality is by the Lipschitz continuity of $g$ and $h$, and the final one is by Fubini's theorem. Using Proposition \ref{thm:convgamma0} combined with $\E_x[\nu_{T-t}]\leq K_2(1+|x|_d)$ from Lemma \ref{lem:3.3} we have that
\begin{align}\label{eq:ugliminf}
u^\gamma(t,x)-\overline{v}(t,x)\geq -K(1+T)\gamma K_2K_3(1+|x|_d)-2\eta.
\end{align}
Taking the liminf as $\gamma\downarrow0$ we get
\begin{align*}
\underline{u}(t,x)-\overline{v}(t,x)\geq-2\eta.
\end{align*}
By the arbitrariness of $\eta$, we obtain $\underline{u}(t,x)\geq \overline{v}(t,x)$ as claimed.

We prove now that $\overline{u}(t,x)\leq \underline{v}(t,x)$. Let $\tau\in\cT_t$ be an $\eta$-optimal stopping time for $u^{\gamma}$, i.e.,
\[
\inf_{(n,\nu)\in\cA^d}\cJ^\gamma_{t,x}(n,\nu,\tau)\ge u^\gamma(t,x)-\eta.
\]
Let $(n,\nu)\in \cA^{d_0}$ be an $\eta$-optimal control for 
\[
\inf_{(\hat{n}, \hat \nu) \in \cA^{d_0}} \cJ_{t,x}(\hat{n}, \hat \nu, \tau),
\] 
corresponding to $\tau$, i.e., $\underline{v}(t,x) \ge \cJ_{t,x}(n, \nu, \tau) - \eta$. Thanks to Lemma \ref{lem:3.3} we can assume without loss of generality that $(n,\nu)\in\cA^{d_0,opt}_{t,x}$. Notice that $(n,\nu)\in \mathcal{A}^{d_0,opt}_{t,x}\subset \cA^d$ is an admissible control in the game with value $u^\gamma$ and, moreover, 
\begin{align}\label{eq:indist}
X_{s}^{[n,\nu],\gamma}=X_{s}^{[n,\nu]}\quad\text{for all $s\in[0,T-t]$, $\P_x$-a.s.}
\end{align} 
Thus, using the above indistinguishability and recalling $f^\gamma\leq f+\sqrt{\gamma}K$ we easily obtain
\begin{align}\label{eq:uglimsup}
u^\gamma(t,x)-\underline{v}(t,x)\leq&\, \cJ_{t,x}^\gamma(n,\nu,\tau)-\cJ_{t,x}(n,\nu,\tau)+2\eta\\
=&\,\E_x\Big[\int_{[0,\tau_\gamma]}\!e^{-rs}\big(f^\gamma(t+s)-f(t+s)\big)\,\ud \nu_s\Big]+2\eta\notag\\
\leq&\, \sqrt{\gamma} K\E_x[\nu_{T-t}] +2\eta\leq \sqrt{\gamma} K K_2(1+|x|_d)+2\eta,\notag
\end{align}
where the final inequality is by Lemma \ref{lem:3.3}. Taking limsup as $\gamma\downarrow 0$ and thanks to the arbitrariness of $\eta$ we get $\overline{u}(t,x)\leq \underline{v}(t,x)$ as claimed. 

As a result, the pointwise limit $\lim_{\gamma\to0} u^\gamma$ is well-defined and $u\coloneqq\lim_{\gamma\to0} u^\gamma=\underline{v}=\overline{v}=v$. Combining \eqref{eq:ugliminf} and \eqref{eq:uglimsup} for $\gamma\in(0,1)$ yields \eqref{eq:convunifutheta_1}.
\end{proof}
Since $|\nabla u^\gamma|_\gamma\le f^\gamma$ for every $\gamma>0$, the next corollary holds.
\begin{corollary}\label{cor:lip}
The value function $v$ is Lipschitz in the first $d_0$ spatial coordinates with constant bounded by $f$, i.e., $|\nabla^0 v(t,x)|_{d_0}\le f(t)$ for a.e.\ $(t,x)\in\R^{d+1}_{0,T}$.
\end{corollary}

For $(n,\nu)\in \cA^{d_0}$, we recall stopping times
\begin{align*}
&\tau_*=\inf\{s\geq 0| v(t+s,X_{s}^{[n,\nu]})-g(t+s,X_{s}^{[n,\nu]})=0\},\\
&\sigma_*=\inf\{s\geq 0| v(t+s,X_{s-}^{[n,\nu]})-g(t+s,X_{s-}^{[n,\nu]})=0\},
\end{align*}
and 
\begin{equation}\label{eqn:theta_star}
\theta_*=\tau_*\wedge\sigma_*.
\end{equation}
\begin{lemma}\label{lem:convth}
Fix $(t,x)\in\R^{d+1}_{0,T}$. For any pair $(n,\nu)\in \cA^{d_0}$ we have
\[
\liminf_{\gamma\downarrow 0}\theta^\gamma_*\ge \theta_*,\qquad\text{$\P_x$-a.s.},
\]
where $\theta^\gamma_*$ is defined in \eqref{eq:thetagamma}.
\end{lemma}
\begin{proof}
Fix $(t,x)\in\R^{d+1}_{0,T}$ and take $(n,\nu)\in \cA^{d_0}$. Let
\[
Z_s=(v-g)(t+s,X^{[n,\nu]}_s)\quad \text{and}\quad Z^\gamma_s=(u^\gamma-g)(t+s,X^{[n,\nu]}_s).
\]
Since $(n,\nu)\in \cA^{d_0}$, we have $X^{[n, \nu]} \equiv X^{[n, \nu], \gamma}$, so $\theta^\gamma_* = \inf\{s \ge 0: \min(Z^\gamma_s, Z^\gamma_{s-}) = 0 \}$. Similarly, $\theta_* = \inf\{s \ge 0: \min(Z_s, Z_{s-}) = 0 \}$.

For $\omega\in\Omega$ such that $\theta_*(\omega)=0$ the claim in the lemma is trivial. Let $\omega\in\Omega$ be such that $\theta_*(\omega)>0$. Take arbitrary $\delta<\theta_*(\omega)$. Then, by the definition of $\theta_*$ we have
\begin{equation}\label{eqn:Z_positive}
\min(Z_s(\omega),Z_{s-}(\omega))>0 \quad\text{for all $s\in[0,\delta]$}.
\end{equation}
Furthermore,
\begin{align}\label{eq:infZ}
\inf_{0\le s\le \delta}\min(Z_s(\omega),Z_{s-}(\omega))=:\lambda_{\delta,\omega} > 0,
\end{align}
as the mapping $s\mapsto \min(Z_s(\omega),Z_{s-}(\omega))$ is lower semi-continuous so it attains its infimum on $[0, \delta]$.

Since $(n,\nu)$ is fixed and $\E_x[\nu_{T-t}^2]<\infty$ by definition of $ \cA^{d_0}$, for almost every $\omega$ there is a compact $K_{\delta,\omega}\subset\R^{d+1}_{0,T}$ that contains the trajectories 
\[
s\mapsto (t+s,X^{[n,\nu]}_s(\omega))\quad \text{and}\quad s\mapsto (t+s,X^{[n,\nu]}_{s-}(\omega))
\]
for $s\in[0,\delta]$. Then, uniform convergence of $u^\gamma$ to $v$ on $K_{\delta,\omega}$ (see \eqref{eq:convunifutheta_1}) yields
\[
\lim_{\gamma\to 0}\sup_{0\le s\le \delta}\big(|Z^\gamma_s(\omega)-Z_s(\omega)|+|Z^\gamma_{s-}(\omega)-Z_{s-}(\omega)|)=0.
\]
Hence, for all sufficiently small $\gamma>0$, \eqref{eq:infZ} yields
\[
\inf_{0\le s\le \delta}\min(Z^\gamma_s(\omega),Z^\gamma_{s-}(\omega))\ge\frac{\lambda_{\delta,\omega}}{2},
\]
which implies 
\[
\liminf_{\gamma\downarrow 0}\theta_*^\gamma(\omega)\ge \delta.
\]
We conclude that 
\[
\liminf_{\gamma\downarrow 0}\theta_*^\gamma(\omega)\ge\theta_*(\omega),
\]
because $\delta$ was arbitrary. The result holds for a.e.\ $\omega$, and the proof is complete.
\end{proof}

We will extract from the uniform convergence of $u^\gamma$ to $v$ and from Lemma \ref{lem:convth} the optimality of the stopping time $\theta_*$. This notion of optimality is discussed in detail in Remark \ref{rem:opt}.
\begin{theorem}\label{thm:opttaustar_1}
Fix $(t,x)\in\R^{d+1}_{0,T}$. We have 
\[
v(t,x) \le \cJ_{t,x}(n,\nu,\theta_*),
\]
for any $(n,\nu)\in \cA^{d_0}$, where we recall that $\theta_*=\theta_*(t,x;n,\nu)$ depends on the initial point and the control pair. Furthermore,
\[
v(t,x)=\inf_{(n,\nu)\in \cA^{d_0}}\cJ_{t,x}\big(n,\nu,\theta_*(t,x;n,\nu)\big),
\] 
hence $\theta_*$ is optimal for the stopper in the game with value $v$.
\end{theorem}
\begin{proof}
We follow an approach inspired by \cite[Thm.\ 4.12]{chiarolla2016hilbert} in optimal stopping. 
Notice that $(n,\nu)\in\cA^{d_0}\subset\cA^{d}$ and $X^{[n,\nu],\gamma}=X^{[n,\nu]}$, i.e., the processes are indistinguishable. Since $\theta_*^\gamma\wedge\theta_*\le\theta^\gamma_*$ it is not difficult to verify that \eqref{eq:ev0} continues to hold when we replace the pair $(t+\theta_*^\gamma,X_{\theta_*^\gamma}^{[n,\nu],\gamma})$ therein by $(t+\theta_*^\gamma\wedge\theta_*,X_{\theta_*^\gamma\wedge\theta_*}^{[n,\nu]})$. That is, we have
\begin{align}\label{eq:uv0}
u^\gamma(t,x)\le \E_x\Big[&e^{-r(\theta^\gamma_*\wedge\theta_*)}u^\gamma\big(t+\theta_*^\gamma\wedge\theta_*,X_{\theta_*^\gamma\wedge\theta_*-}^{[n,\nu]}\big)\!+\!\int_{0}^{\theta_*^\gamma\wedge\theta_*}\!\!\!e^{-rs}h(t+s,X_s^{[n,\nu]})\ud s \notag\\
&-\int_{0}^{\theta_*^\gamma\wedge\theta_*}\!\!e^{-rs}\langle\nabla u^\gamma(t\!+\!s,X_{s-}^{[n,\nu]}), n_s\rangle_\gamma\,\ud \nu_s^c\\
&\!-\! \sum_{s<\theta_*^\gamma\wedge\theta_*}\!\!e^{-rs}\!\!\int_{0}^{\Delta\nu_{s}}\!\!\langle\nabla u^\gamma(t\!+\!s,X_{s-}^{[n,\nu]}\!+\!\lambda n_s),n_s\rangle_\gamma\, \ud\lambda \Big].\notag
\end{align}
Further using that $|\nabla u^\gamma|_\gamma\le f^\gamma$ leads to 
\begin{align}\label{eq:convthetagam}
u^\gamma(t,x)\le&\, \E_x\Big[e^{-r(\theta^\gamma_*\wedge\theta_*)}u^\gamma\big(t+\theta_*^\gamma\wedge\theta_*,X_{\theta_*^\gamma\wedge\theta_*-}^{[n,\nu]}\big)\!+\!\int_{0}^{\theta_*^\gamma\wedge\theta_*}\!\!\!e^{-rs}h(t+s,X_s^{[n,\nu]})\ud s \notag\\
&\qquad+\int_{[0,\theta_*^\gamma\wedge\theta_*)}\!\!e^{-rs}f^\gamma(t+s)\,\ud \nu_s\Big]\\
\le&\, \E_x\Big[e^{-r(\theta^\gamma_*\wedge\theta_*)}v\big(t+\theta_*^\gamma\wedge\theta_*,X_{\theta_*^\gamma\wedge\theta_*-}^{[n,\nu]}\big)\!+\!\int_{0}^{\theta_*^\gamma\wedge\theta_*}\!\!\!e^{-rs}h(t+s,X_s^{[n,\nu]})\ud s \notag\\
&\qquad+\int_{[0,\theta_*^\gamma\wedge\theta_*)}\!\!e^{-rs}f^\gamma(t+s)\,\ud \nu_s\Big]+C\gamma^{1/2}\Big(1+\E_x\big[\big|X^{[n,\nu]}_{\theta^\gamma_*\wedge\theta_*}\big|_d\big]\Big),\notag
\end{align}
where in the second inequality we used \eqref{eq:convunifutheta_1}.

We now let $\gamma\downarrow 0$ and notice that $\theta^\gamma_*\wedge\theta_*\to \theta_*$ by Lemma \ref{lem:convth}. Since the mappings
\[
s\mapsto X^{[n,\nu]}_{s-}\quad\text{and}\quad s\mapsto \int_{[0,s)}e^{-ru}f(t+u)\ud \nu_u
\]
are left-continuous $\P_x$-a.s.\ and $\theta_*^\gamma\wedge\theta_*$ converges to $\theta_*$ from below (although not strictly from below), we can conclude that for a.e.\ $\omega\in\Omega$
\begin{align}\label{eq:leftcont}
\lim_{\gamma\to 0}X^{[n,\nu]}_{\theta_*^\gamma\wedge\theta_*-}=X^{[n,\nu]}_{\theta_*-}\quad\text{and}\quad\lim_{\gamma\to 0}\int_{[0,\theta_*^\gamma\wedge\theta_*)}e^{-rs}f^\gamma(t+s)\ud \nu_s=\int_{[0,\theta_*)}e^{-rs}f(t+s)\ud \nu_s.
\end{align}
Moreover, we can use dominated convergence thanks to, e.g., Corollary \ref{lem:L1estcntrSDE} and the identification $\cA^{d_0}=\cA^{d_0,opt}_{t,x}$ following Lemma \ref{lem:3.3}. That yields
\begin{align}\label{eq:vopt}
v(t,x)\le \E_x\Big[e^{-r\theta_*}v\big(t+\theta_*,X_{\theta_*-}^{[n,\nu]}\big)\!+\!\int_{0}^{\theta_*}\!\!\!e^{-rs}h(t+s,X_s^{[n,\nu]})\ud s +\int_{[0,\theta_*)}\!\!e^{-rs}f(t+s)\,\ud \nu_s\Big].
\end{align}

We now follow similar arguments as in the proof of Lemma \ref{lem:convst} (below Eq.~\eqref{eq:ev1}) to show that on the event $\{\sigma_*<\tau_*\}$ we have
$v\big(t+\sigma_*,X_{\sigma_*-}^{[n,\nu]}\big)=g\big(t+\sigma_*,X_{\sigma_*-}^{[n,\nu]}\big)$, so 
\begin{align}\label{eq:sigma<tau}
v\big(t+\theta_*,X_{\theta_*-}^{[n,\nu]}\big)=&\,v\big(t+\sigma_*,X_{\sigma_*-}^{[n,\nu]}\big)=g\big(t+\sigma_*,X_{\sigma_*-}^{[n,\nu]}\big)\notag\\
=&\,g\big(t+\sigma_*,X_{\sigma_*}^{[n,\nu]}\big)-\int_0^{\Delta\nu_{\sigma_*}}\langle\nabla^0 g(t+\sigma_*,X^{[n,\nu]}_{\sigma_*-}+\lambda n_{\sigma_*}),n_{\sigma_*}\rangle\ud\lambda\\
\le&\,g\big(t+\sigma_*,X_{\sigma_*}^{[n,\nu]}\big)+f(t+\sigma_*)\Delta\nu_{\sigma_*},\notag
\end{align}
where the inequality uses the bound on the gradient $\nabla^0 g$ imposed by Assumption \ref{ass:gen2}(iii).

On the event $\{\sigma_*\ge \tau_*\}$ we have
\begin{align}\label{eq:tau<sigma}
v\big(t+\theta_*,X_{\theta_*-}^{[n,\nu]}\big) &= v\big(t+\tau_*,X_{\tau_*-}^{[n,\nu]}\big)\notag\\
&=v\big(t+\tau_*,X_{\tau_*}^{[n,\nu]}\big)-\int_0^{\Delta\nu_{\tau_*}}\langle\nabla^0 v(t+\tau_*,X^{[n,\nu]}_{\tau_*-}+\lambda n_{\tau_*}),n_{\tau_*}\rangle\ud\lambda\\
&\le v\big(t+\tau_*,X_{\tau_*}^{[n,\nu]}\big)+f(t+\tau_*)\Delta\nu_{\tau_*},\notag\\
&= g\big(t+\tau_*,X_{\tau_*}^{[n,\nu]}\big)+f(t+\tau_*)\Delta\nu_{\tau_*},\notag
\end{align}
where the inequality uses the bound on the gradient $\nabla^0 v$ from Corollary \ref{cor:lip} and the last equality holds by the definition of $\tau_*$ and the right-continuity of the  process $X^{[n,\nu]}$.
 
Combining the upper bounds above with \eqref{eq:vopt} yields
$v(t,x)\le \cJ_{t,x}(n,\nu,\theta_*)$,
where we emphasise that $\theta_*$ depends on the control $(n,\nu)$. By arbitrariness of $(n,\nu)\in\cA^{d_0}$ we can conclude the proof of the theorem because $v = \overline{v}$ implies $v(t,x)\ge \inf_{(n,\nu)\in \cA^{d_0}}\cJ_{t,x}(n,\nu,\theta_*)$.
\end{proof}

\subsection{Relaxing Assumption \ref{ass:gen3} into Assumption \ref{ass:gen2}}\label{sec:relax41}
In this section we prove Theorem \ref{thm:usolvar} via a localisation and mollification procedure, and using the results from the section above.
For technical reasons, we assume first that the functions $g$ and $h$ are uniformly bounded, i.e., $\|g\|_{\infty}+\|h\|_\infty<\infty$, and then we relax this condition in the second part of the proof.

Fix a compact set $\hat \Sigma \subset \R^d$ and denote $\Sigma = [0, T] \times \hat \Sigma$. There is a family $(\zeta_j)_{j\in\N} = (\zeta_j^\Sigma)_{j\in\N}$ of mollifiers in $\R^{d+1}_{0,T}$ and a sequence $(c_j)_{j\in\N}$ of positive numbers converging to $0$ such that, denoting $g^j\coloneqq g*\zeta_j$, $h^j\coloneqq h*\zeta_j$ and $f^j\coloneqq(f+c_j)*\zeta_j$, we have
\begin{align}\label{eq:unifc}
\|g^j-g\|_{C^0(\Sigma)}+\|h^j-h\|_{C^0(\Sigma)}\leq \tfrac{K_\Sigma}{j}\quad\text{for any $j\in\N$ and a constant $K_\Sigma > 0$},
\end{align} 
and
\begin{equation}\label{eq:unifc1}
0\leq f^j-f\leq \tfrac{c_0}{j},\quad\text{for any $j\in\N$ and a constant $c_0>0$}.
\end{equation}
Note that the definition of $f^j$ is with an abuse of notation as $f$ depends only on $t$: for this mollification we extend $f$ into the spatial dimension as a constant function.

Recall that $B_k\subset\R^d$ denotes the ball of radius $k$ centred in the origin. Let $(\xi_k)_{k\in\N}\subset C^{\infty}_c(\R^{d})$ be a sequence of cut-off functions such that $\xi_k(x)=1$ for $x\in B_k$ and $\xi_k(x)=0$ for $x\notin B_{2k}$. We find it convenient to construct the sequence as follows: let 
\begin{align}\label{eq:xi}
\xi(z)\coloneqq\left\{
\begin{array}{ll}
1& z\le 0,\\
0& z\ge 1,\\
\exp\big(\tfrac{1}{z-1}\big)/\big[\exp\big(\tfrac{1}{z-1}\big)+\exp\big(-\tfrac1z\big)\big],& z\in(0,1),
\end{array}
\right.
\end{align}
so that $\|\xi'\|_\infty=2$, and define $\xi_k(x)\coloneqq \xi(\tfrac{|x|_{d}-k}{k})$ for $x\in\R^d$. Then
\begin{align}\label{eq:gradxi}
|\nabla^0\xi_k(x)|^2_{d_0} \leq |\nabla \xi_k(x)|_d^2=\tfrac{1}{k^2}|\xi'(\tfrac{|x|_d-k}{k})|^2\leq \tfrac{4}{k^2}.
\end{align}

Now we set $g^{j,k}\coloneqq g^{j}\xi_k $, $h^{j,k}\coloneqq h^j\xi_k$, and 
\begin{align*}
f^{j,k}(t)\coloneqq f^j(t)+\tfrac{2}{k}\|g\|_{\infty},
\end{align*}
where we construct $g^j$, $h^j$ and $f^j$ using the above mollification procedure with $\hat\Sigma=B_{k}$.
With such choice of $f^{j,k}$, recalling that $|\nabla^0g^j|_{d_0}\le f^j$ by (iii) in Assumption \ref{ass:gen2} and using the bound in \eqref{eq:gradxi}, we have
\begin{align*}
&|\nabla^0g^{j,k}(t,x)|_{d_0}^2\\
&=(\xi_k(x))^2|\nabla^0g^j(t,x)|_{d_0}^2+2\xi_k(x)g^{j}(t,x)\langle \nabla^0 g^j(t,x),\nabla^0\xi_k(x)\rangle+(g^j(t,x))^2|\nabla^0\xi_k(x)|_{d_0}^2\\
&\leq (f^j(t))^2+4\|g\|_{\infty}f^j(t)/k+4\|g\|^2_\infty/k^2\\
&=(f^j(t)+\tfrac{2}{k}\|g\|_\infty)^2=\big(f^{j,k}(t)\big)^2.
\end{align*}

Fix $(t,x)\in\R^{d+1}_{0,T}$. For an arbitrary treble $[(n,\nu),\tau]\in \cA^{d_0}\times\cT_t$ we consider the game with expected payoff 
\begin{align}\label{eq:Jbdd}
&\cJ^{j,k}_{t,x}(n,\nu,\tau)\\
&= \E_{x}\Big[e^{-r\tau}\!g^{j,k}(t\!+\!\tau,X_\tau^{[n,\nu]})\!+\!\int_0^{\tau}\!\! e^{-rs}h^{j,k}(t\!+\!s,X_s^{[n,\nu]})\,\ud s\! +\!\int_{[0,\tau]}\!\! e^{-rs}f^{j,k}(t\!+\!s)\,\ud \nu_s \Big].\notag
\end{align}
Since $f^j\in C^\infty([0,T])$, $g^{j,k},h^{j,k}\in C^\infty_{c,\,\rm sp}(\R^{d+1}_{0,T})$, then Theorems \ref{thm:valgame_1} and \ref{thm:opttaustar_1} yield that there exists a value $v^{j,k}$ of the game and an optimal stopping time $\theta_*^{j,k}=\tau^{j,k}_*\wedge\sigma^{j,k}_*$, with
\begin{align*}
\begin{split}
&\tau_*^{j,k}\coloneqq \inf\big\{s\geq 0\,\big|\, v^{j,k}(t+s,X_s^{[n,\nu]})=g^{j,k}(t+s,X_s^{[n,\nu]})\big\},\\
&\sigma_*^{j,k}\coloneqq \inf\big\{s\geq 0\,\big|\, v^{j,k}(t+s,X_{s-}^{[n,\nu]})=g^{j,k}(t+s,X_{s-}^{[n,\nu]})\big\}.
\end{split}
\end{align*}
Finally, we set
\[
v^{\infty}\coloneqq \limsup_{k\to\infty}\limsup_{j\to\infty}v^{j,k}
\]
and proceed to show that $v^{\infty}\ge \overline v$ and $v^\infty\le \underline v$.

\begin{lemma}\label{lem:uinfty}
Let Assumptions \ref{ass:gen1} and \ref{ass:gen2} hold and assume $\|g\|_\infty+\|h\|_\infty<\infty$. For any $(t,x)\in\R^{d+1}_{0,T}$ we have 
\[
v^\infty(t,x)=\overline v(t,x)=\underline v(t,x),
\]
hence the value $v$ of the game \eqref{eq:valuegame_1} exists.
\end{lemma}
\begin{proof}
We start by proving $v^{\infty}\leq \underline{v}$. Take $\theta^{j,k}_*$ defined above. Then
\[
\underline v(t,x)\ge \inf_{(n,\nu)\in \cA^{d_0}}\cJ_{t,x}(n,\nu,\theta^{j,k}_*),
\]
as $\theta^{j,k}_*$ is suboptimal for $\underline v$. For any
$\eta>0$ there is a pair $(n^{j,k,\eta},\nu^{j,k,\eta})$ such that
\[
\inf_{(n,\nu)\in \cA^{d_0}}\cJ_{t,x}(n,\nu,\theta^{j,k}_*)\ge\cJ_{t,x}(n^{j,k,\eta},\nu^{j,k,\eta},\theta^{j,k}_*)-\eta. 
\]
Moreover, from the optimality of $\theta_*^{j,k}$ for $v^{j,k}$ in the sense of Theorem \ref{thm:opttaustar_1}, we have
\[
v^{j,k}(t,x)\le \cJ^{j,k}_{t,x}(n^{j,k,\eta},\nu^{j,k,\eta},\theta^{j,k}_*).
\]

For ease of notation we denote $(\theta^{j,k}_*,n^{j,k,\eta},\nu^{j,k,\eta})=(\theta,n,\nu)$ in what follows. Combining the two bounds above and recalling that $g^{j,k}=g^j$ and $h^{j,k}=h^j$ in $[0,T]\times B_k$ we obtain
\begin{equation*}
\begin{aligned}
v^{j,k}(t,x)-\underline{v}(t,x)&\leq\cJ^{j,k}_{t,x}(n,\nu,\theta)-\cJ_{t,x}(n,\nu,\theta)+\eta \\
&\le \E_x\Big[\big|g^{j}\big(t+\theta,X_{\theta}^{[n,\nu]}\big)-g\big(t+\theta,X_{\theta}^{[n,\nu]}\big)\big|\mathds{1}_{\{X_{\theta}^{[n,\nu]} \in B_{k}\}}+2e^{-r\theta}\|g\|_\infty\mathds{1}_{\{X_{\theta}^{[n,\nu]} \notin B_{k}\}}\Big]\\
&\quad+\E_x\Big[\int_0^{\theta}\!\big|h^{j,k}\big(t+s,X_s^{[n,\nu]}\big)-h\big(t+s,X_s^{[n,\nu]}\big)\big|\mathds{1}_{\{X_{s}^{[n,\nu]} \in B_{k}\}}\,\ud s\Big]\\
&\quad+\E_x\Big[2\|h\|_\infty\int_0^{\theta}\!e^{-rs}\mathds{1}_{\{X_{s}^{[n,\nu]} \notin B_{k}\}}\,\ud s+\nu_{T-t}\big(\tfrac{c_0}{j}+\tfrac{2}{k}\|g\|_{\infty}\big)\Big]+\eta,\end{aligned}
\end{equation*}
where we used $f^{j,k}-f\le c_0/j+2\|g\|_\infty/k$. Next, from \eqref{eq:unifc} we obtain 
\begin{equation}\label{eq:approx0}
\begin{aligned}
&v^{j,k}(t,x)-\underline{v}(t,x)\\
&\leq(1+T)\tfrac{K_{B_{k}}}{j}+(\tfrac{c_0}{j}+\tfrac{2}{k}\|g\|_{\infty})\E_x\big[\nu_{T-t}\big]+2\|g\|_\infty\P_x\big(X_{\theta}^{[n,\nu]} \notin B_{k}\big)\\
&\quad+2\|h\|_\infty\int_0^{T-t}\P_x\big(X_{s}^{[n,\nu]} \notin B_{k}\big)\ud s+\eta,
\end{aligned}
\end{equation}

From the proof of Lemma \ref{lem:3.3}, we can restrict our attention to processes $(n,\nu) \in \cA^{d_0,opt}_{t,x}$ for which $\E[\nu_{T-t}]\le c(1+|x|_{d})$, where the constant $c$ can be chosen independently of $(j,k)$. From Corollary \ref{lem:L1estcntrSDE} and Markov's inequality we deduce that 
\[
\P_x\big(X_{\theta}^{[n,\nu]} \notin B_{k}\big)\le \frac1k\E_x\big[\big|X_{\theta}^{[n,\nu]}\big|_d\big] \le \frac{\tilde{K}_4(1+|x|_d)}{k}
\]
for some $\tilde{K}_4 > 0$ and the same upper bound also holds for $\P_x(X_{s}^{[n,\nu]} \notin B_{k})$. Now, letting $j\to\infty$ first and then letting $k\to \infty$ in \eqref{eq:approx0} we obtain $v^\infty\le \underline v + \eta$. Finally we send $\eta \to 0$. 

Next we are going to show that $v^\infty\ge \overline v$. Fix an arbitrary $\eta>0$. Take $(n^{j,k,\eta},\nu^{j,k,\eta})\in \cA^{d_0}$ such that
\[
v^{j,k}(t,x) \ge\sup_{\tau\in\cT_t} \cJ^{j,k}_{t,x}(n^{j,k,\eta},\nu^{j,k,\eta},\tau) - \eta.
\]
Then there is $\tau^{j,k,\eta}$ such that
\[
\overline{v}(t,x) \le \sup_{\tau \in \cT_t} \cJ_{t,x}(n^{j,k,\eta},\nu^{j,k,\eta}, \tau) \le \cJ_{t,x}(n^{j,k,\eta},\nu^{j,k,\eta}, \tau^{j,k,\eta}) + \eta,
\]
where the first inequality follows from suboptimality of $(n^{j,k,\eta},\nu^{j,k,\eta})$ for $\overline{v}(t,x)$. Relabelling 
\[
[(n^{j,k,\eta},\nu^{j,k,\eta}),\tau^{j,k,\eta}]=[(n,\nu),\tau],
\] 
the above inequalities give the bound
\[
\overline v(t,x)-v^{j,k}(t,x)\le \cJ_{t,x}(n,\nu,\tau)-\cJ^{j,k}_{t,x}(n,\nu,\tau)+2\eta.
\]
Similar estimates as in \eqref{eq:approx0} continue to hold, with a simplification that the inequality $f^{j,k}\ge f$ allows us to drop the second term in the final expression therein. Then, passing to the limit in $j$ and, then, in $k$ we arrive at the desired conclusion.
\end{proof}
\begin{remark}\label{rem:uinfty}
Notice that if we introduce
\[
\hat v^\infty\coloneqq\liminf_{k\to\infty}\liminf_{j\to\infty} v^{j,k},
\]
then we can repeat the same arguments of proof as in Lemma \ref{lem:uinfty} to show that $\hat v^\infty=\overline v=\underline v=v$. Hence,
\[
v=\overline v=\underline v=\lim_{k\to\infty}\lim_{j\to\infty}v^{j,k}.
\]
Moreover, it is clear from the proof (see in particular \eqref{eq:approx0}) that the convergence is uniform on any compact subset of $\R^{d+1}_{0,T}$. This fact will be used later to prove convergence of optimal stopping times.
\end{remark}

We now want to extend the result above to the case of unbounded $g$ and $h$. Recalling that $g,h\ge 0$, we can approximate them with bounded ones by setting $g_m=g\wedge m$ and $h_m=h\wedge m$ for $m\in\N$. Let us denote by $v_m$ the value of the game associated with the functions $g_m$ and $h_m$, which exists by Lemma \ref{lem:uinfty}. By construction $v_m\le v_{m+1}$ and we denote the limit
\[
v_\infty\coloneqq\lim_{m\to\infty}v_m .
\]
\begin{lemma}\label{eq:vsolgame}
Let Assumptions \ref{ass:gen1} and \ref{ass:gen2} hold. For any $(t,x)\in\R^{d+1}_{0,T}$ we have 
\[
v_\infty(t,x)=\overline v(t,x)=\underline v(t,x),
\]
hence the value $v$ of the game \eqref{eq:valuegame_1} exists.
\end{lemma}
\begin{proof}
Since $h_m\le h$ and $g_m\le g$, it is immediate to verify that $v_\infty\le\underline v$. It remains to verify that $v_\infty\ge \overline v$. 

Thanks to (sub)linear growth of $g$ and $h$ there is a sequence $(R(m))_{m\in\N}$ such that $R(m)\uparrow \infty$ as $m\to\infty$ and $g_m=g$ and $h_m=h$ on $[0,T]\times B_{R(m)}$. Let us denote by $\cJ^m_{t,x}$ the expected payoff of the game with the payoff functions $g_m$ and $h_m$.
For fixed $\eta>0$, we can find a pair $(n,\nu)=(n^{m,\eta},\nu^{m,\eta})\in \cA^{d_0}$ and a stopping time $\tau=\tau^{m,\eta}\in\cT_t$ such that 
\begin{align*}
&\,\overline{v}(t,x)-v_{m}(t,x)\\
&\leq \cJ_{t,x}(n,\nu,\tau)-\cJ^{m}_{t,x}(n,\nu,\tau)+2\eta\\
&\leq \E_x\Big[g(t\!+\!\tau,X_{\tau}^{[n,\nu]})\mathds{1}_{\{X^{[n,\nu]}_\tau\notin B_{R(m)} \}}\!+\!\int_{0}^{T-t}e^{-rs}\mathds{1}_{\{X^{[n,\nu]}_s\notin B_{R(m)} \}}h(t\!+\!s,X^{[n,\nu]}_s)\ud s\Big]\!+\!2\eta,
\end{align*}
where we obtained the inequality simply by dropping the positive terms $h_m$ and $g_m$ on the events when the process is outside the ball $B_{R(m)}$. 

By the strict sub-linear growth of $g$ and $h$ and using H\"older's inequality we obtain
\begin{equation}\label{eq:vmlim}
\begin{aligned}
&\overline{v}(t,x)\!-\!v_{m}(t,x)-2\eta\\
&\leq \E_x\Big[K_1\big(1\!+\!\big|X_{\tau}^{[n,\nu]}\big|^\beta_d\big)\mathds{1}_{\{X^{[n,\nu]}_\tau\notin B_{R(m)}\}}\!+\!\int_{0}^{T-t}\!\!K_1\big(1\!+\!\big|X^{[n,\nu]}_s\big|^\beta_d\big)\mathds{1}_{\{X^{[n,\nu]}_s\notin B_{R(m)} \}}\ud s\Big]\\
&\le K_1\Big\{\big(\E_x\big[\big|X_{\tau}^{[n,\nu]}\big|_d\big]\big)^\beta \big(\P_x (X^{[n,\nu]}_\tau\notin B_{R(m)})\big)^{1-\beta}\\
&\hspace{120pt}+\!\int_0^{T-t}\!\!\big(\E_x\big[\big|X_{s}^{[n,\nu]}\big|_d\big]\big)^\beta \big(\P_x(X^{[n,\nu]}_s\notin B_{R(m)})\big)^{1-\beta}\ud s\Big\}.
\\
&\quad+\!K_1\P_x\big(X^{[n,\nu]}_\tau\notin B_{R(m)}\big)\!+\!K_1\int_0^{T-t}\P_x\big(X^{[n,\nu]}_s\notin B_{R(m)}\big)\ud s.
\end{aligned}
\end{equation}
For any $\sigma\in\cT_t$, Markov's inequality and an estimate for $\E_x[|X_{\sigma}^{[n,\nu]}|_d]$ from Corollary \ref{lem:L1estcntrSDE} give
\[
\P_x\big(X^{[n,\nu]}_\sigma\notin B_{R(m)}\big)\le \frac{\E_x\big[\big|X_{\sigma}^{[n,\nu]}\big|_d\big]}{R(m)}\le \frac{K_4(1+|x|_d)}{R(m)}. 
\]
Therefore, letting $m\to\infty$ in \eqref{eq:vmlim} we find $\overline v\le v_\infty+2\eta$ and, by arbitrariness of $\eta>0$, we conclude the proof.
\end{proof}
\begin{remark}\label{rem:unifv_m}
As in Lemma \ref{lem:uinfty}, also in the lemma above the convergence of $v_m$ to $v$ is uniform on compact subsets of $\R^{d+1}_{0,T}$. This is immediately deduced from \eqref{eq:vmlim} and the concluding estimates in the proof. 
\end{remark}
Since $|\nabla^0 v^{j,k}|_{d_0}\le f^{j,k}$ a.e.\ for every $j,k>0$ (Corollary \ref{cor:lip}), then $|\nabla^0 v^\infty|_{d_0}\le f$ a.e. By the same rationale, also $v_\infty$ satisfies the same bound. That is stated formally in the next corollary.
\begin{corollary}\label{cor:lip2}
Under Assumptions \ref{ass:gen1} and \ref{ass:gen2}, the value function $v$ is Lipschitz in the first $d_0$ spatial coordinates with constant bounded by $f$ in the sense that $|\nabla^0 v(t,x)|_{d_0}\le f(t)$ for a.e.\ $(t,x)\in\R^{d+1}_{0,T}$.
\end{corollary}

The last result in this section concerns an optimal stopping time for the value $v=v_\infty$. Given $(n,\nu)\in \cA^{d_0}$, set $\theta_*= \tau_*\wedge\sigma_*$ as in \eqref{eqn:theta_star}.

\begin{lemma}\label{lem:5.10}
Let Assumptions \ref{ass:gen1} and \ref{ass:gen2} hold. For any $(t,x)\in\R^{d+1}_{0,T}$ and $(n,\nu)\in \cA^{d_0}$ we have 
\[
v(t,x) \le \cJ_{t,x}(n,\nu,\theta_*),
\] 
where we recall that $\theta_*=\theta_*(n,\nu)$ depends on the control pair. Hence 
\[
v(t,x)=\inf_{(n,\nu)\in \cA^{d_0}}\cJ_{t,x}(n,\nu,\theta_*(n,\nu))
\] 
and $\theta_*$ is optimal for the stopper in the game with value $v$.
\end{lemma}

\begin{proof}
The proof follows similar arguments as those used in the proof of Thorem \ref{thm:opttaustar_1}, so we provide only a sketch.
Let $g$ and $h$ be functions that satisfy Assumption \ref{ass:gen2}. For $m\in\N$, consider their truncations $g_{m}(t,x)= g(t,x)\wedge m$ and $h_{m}(t,x)= h(t,x)\wedge m$. 
We further mollify and localise those functions to fit into the setting of Theorem \ref{thm:opttaustar_1} as in the beginning of Section \ref{sec:relax41}. Denote
\begin{align*}
f^{j,k}_m(t)\coloneqq &\,((f+c^{m,k}_j)*\zeta^{m,k}_j)(t)+\tfrac{2m}{k},\\
g^{j,k}_m(t,x)\coloneqq&\, (g_m*\zeta^{m,k}_j)(t,x)\xi_k(x),\\
h^{j,k}_m(t,x)\coloneqq&\, (h_m*\zeta^{m,k}_j)(t,x)\xi_k(x),
\end{align*}
where $(\zeta^{m,k}_j)_{j\in\N}$ is a sequence of standard mollifiers and $c^{m,k}_j$ is a sequence of positive numbers so that estimates \eqref{eq:unifc}-\eqref{eq:unifc1} hold, and the cut-off functions $(\xi_k)_{n\in\N}$ are obtained with the construction in \eqref{eq:xi}.

Denote by $v^{j,k}_m$ the value function of the game with payoff functions $f^{j,k}_m$, $g^{j,k}_m$ and $h^{j,k}_m$ (c.f. \eqref{eq:Jbdd}). An optimal stopping time for this game is $\theta_*^{j,k,m}\coloneqq \tau_*^{j,k,m}\wedge\sigma_*^{j,k,m}$ with
\begin{align*}
\begin{split}
&\tau_*^{j,k,m}\coloneqq \inf\big\{s\geq 0\,\big|\, v^{j,k}_m(t+s,X_s^{[n,\nu]})=g^{j,k}_m(t+s,X_s^{[n,\nu]})\big\},\\
&\sigma_*^{j,k,m}\coloneqq \inf\big\{s\geq 0\,\big|\, v^{j,k}_m(t+s,X_{s-}^{[n,\nu]})=g^{j,k}_m(t+s,X_{s-}^{[n,\nu]})\big\},
\end{split}
\end{align*}
for an arbitrary pair $(n,\nu)\in \cA^{d_0}$.
The same arguments as in the proof of Lemma \ref{lem:convth} and the uniform convergence of $v^{j,k}_m$ to $v$ on compact subsets of $\R^{d+1}_{0,T}$ (Lemmas \ref{lem:uinfty} and \ref{eq:vsolgame}, and Remarks \ref{rem:uinfty} and \ref{rem:unifv_m}) yield 
\begin{align*}
\liminf_{m\to\infty}\liminf_{k\to\infty}\liminf_{j\to\infty}\theta^{j,k,m}_*\ge \theta_*,\quad \text{$\P_x$-a.s.}
\end{align*}
Hence,
\begin{align}\label{eq:thetajkm-}
\lim_{m\to\infty}\lim_{k\to\infty}\lim_{j\to\infty}\theta_*^{j,k,m}\wedge \theta_*= \theta_*,\quad \text{$\P_x$-a.s.}
\end{align}

The functions $f^{j,k}_m$, $g^{j,k}_m$, $h^{j,k}_m$ satisfy the assumptions of Theorem \ref{thm:opttaustar_1}. Then, by the same arguments as in the proof of that theorem, replacing $u^\gamma$, $\theta^\gamma_*$, $f^\gamma$, $g$ and $h$ by $v^{j,k}_m$, $\theta_*^{j,k,m}$, $f^{j,k}_m$, $g^{j,k}_m$ and $h^{j,k}_m$, respectively, we obtain
\begin{equation}\label{eqn:vjkm}
\begin{aligned}
v^{j,k}_m(t,x)\le \E_x\Big[&\,e^{-r(\theta_*^{j,k,m}\wedge\theta_*)}v^{j,k}_m\Big(t+\theta_*^{j,k,m}\wedge\theta_*,X_{\theta_*^{j,k,m}\wedge\theta_*-}^{[n,\nu]}\Big)\\
&\,+\int_{0}^{\theta_*^{j,k,m}\wedge\theta_*}\!e^{-rs}h^{j,k}_m(t+s,X_s^{[n,\nu]})\ud s+\int_{[0,\theta_*^{j,k,m}\wedge\theta_*)}\!\!e^{-rs}f^{j,k}_m(t+s)\,\ud \nu_s\Big].
\end{aligned}
\end{equation}
We pass to the limit as $j\to\infty$, $k\to\infty$ and $m\to\infty$ (with the limits taken in the stated order). Using \eqref{eq:thetajkm-} and similar arguments as in \eqref{eq:leftcont}, we have that $\P_x$-a.s.
\[
\lim_{j,k,m\to \infty}X^{[n,\nu]}_{\theta_*^{j,k,m}\wedge\theta_*-}=X^{[n,\nu]}_{\theta_*-}\quad\text{and}\quad\lim_{j,k,m\to\infty}\int_{[0,\theta_*^\gamma\wedge\theta_*)}e^{-rs}f^{j,k,m}(t+s)\ud \nu_s=\int_{[0,\theta_*)}e^{-rs}f(t+s)\ud \nu_s.
\]
We apply dominated convergence theorem to \eqref{eqn:vjkm} justified by the linear growth of all functions involved and the fact that one can restrict the attention to controls $(n,\nu)\in \cA^{d_0}$ such that 
\[\E_x[\nu_{T-t}] \le \frac{c(1 + |x|_d)}{\inf_{m,j,k} f^{j,k}_m(T)}\le \tilde c (1 + |x|_d)\] 
for some $\tilde c < \infty$ (c.f. Lemma \ref{lem:3.3} and arguments leading to \eqref{eq:vopt}). In the limit we obtain 
\begin{align}
v(t,x)\le&\,\E_x\Big[e^{-r\theta_*}v\big(t+\theta_*,X_{\theta_*-}^{[n,\nu]}\big)+\int_{0}^{\theta_*}\!e^{-rs}h(t+s,X_s^{[n,\nu]})\ud s +\int_{[0,\theta_*)}\!\!e^{-rs}f(t+s)\,\ud \nu_s\Big].
\end{align}
Corollary \ref{cor:lip2} and ideas from \eqref{eq:sigma<tau} and \eqref{eq:tau<sigma} yield
\begin{align*}
v(t,x)\leq \cJ_{t,x}(n,\nu,\theta_*(n,\nu)),
\end{align*}
which concludes the proof.
\end{proof}

We combine results from this section to prove Theorem \ref{thm:usolvar}.

\begin{proof}[Proof of Theorem \ref{thm:usolvar}] 
Lemma \ref{eq:vsolgame} shows that the game with expected payoff \eqref{eq:payoff} admits a value function $v$. The optimality of the stopping time $\theta_*$ is asserted in Lemma \ref{lem:5.10}. The continuity of the value function $v$ follows from the continuity of $v^{j,k}_m$ in the proof of Lemma \ref{lem:5.10} (from Theorem \ref{thm:usolvar_ga}) and the uniform convergence of $v^{j,k}_m$ to $v$ on compact sets (see Remarks \ref{rem:uinfty} and \ref{rem:unifv_m}). Corollary \ref{cor:lip2} implies the Lipschitz continuity of $v$ in the first $d_0$ spatial coordinates in the sense required in the statement of the theorem. Finally, the growth condition is easily deduced from Assumption \ref{ass:gen2}(ii) and the uniform bound from Corollary \ref{lem:L1estcntrSDE}. 
\end{proof}
\appendix

\section{}

\subsection*{Proof of Theorem \ref{thm:usolvar_ga}} This proof repeats almost verbatim the one from \cite{bovo2022variational}. Here we only summarise the main steps and highlight a few minor changes. 

The game considered here satisfies \cite[Ass.\ 3.1 and 3.2]{bovo2022variational} with the only exception that the dynamics of $X^{[n,\nu],\gamma}$ in \eqref{eq:SDEgam} has a weight $\gamma$ in the last $d_1$ coordinates of the control process $(n_t)_{t\in[0,T]}$. Such weight does not appear in the controlled state-dynamics in \cite{bovo2022variational} but we will show that it only induces minor changes to the proof.

Following the methodology in \cite{bovo2022variational} we study the penalised problem
\begin{align}\label{eq:penalised}
\partial_t u+\cL u-ru=-h-\tfrac1\delta(g-u)^++\psi_\eps\big(|\nabla u|_\gamma^2-(f^\gamma)^2\big),\quad \text{on $[0,T)\times\R^d$},
\end{align}
with terminal condition $u(T,x)=g(T,x)$ and growth condition $|u(t,x)|\le c(1+|x|_d)$ (actually it is enough to consider bounded $u$; see Remark \ref{rem:bddug}). In \cite{bovo2022variational} the penalised problem is first solved on bounded domains, via a localisation procedure, and then on unbounded domain by passing to the limit in the size of the bounded domains. Notice that the parameter $\gamma$ leads to a different penalisation term in \eqref{eq:penalised} compared to \cite{bovo2022variational}. Indeed, in \cite[Eqs.\ (4.3) and (5.14)]{bovo2022variational} the penalisation reads $\psi_\eps(|\nabla u|^2_d-(f^\gamma)^2)$, whereas here we use the $\gamma$-norm $|\cdot|_\gamma$ of the gradient $\nabla u$. 

By Assumption \ref{ass:gen3}(i), for all sufficiently large $m\in\N$ we have $g,h\equiv 0$ on $[0,T]\times(\R^d\setminus B_m)$ and therefore the localisation performed in \cite[Sec.\ 4.1]{bovo2022variational} is superfluous. In particular, we do not need here to introduce functions $f_m$, $g_m$, $h_m$ from \cite{bovo2022variational}. The penalised problem on a bounded domain $[0, T] \times B_m$ is given by
\begin{align}\label{eq:penalisedb}
\left\{
\begin{array}{l}
\partial_t u+\cL u-ru=-h-\tfrac1\delta(g-u)^++\psi_\eps\big(|\nabla u|_\gamma^2-(f^\gamma)^2\big), \quad\text{on $[0,T)\times B_m$},\\
u(t,x)=0,\quad \text{for $(t,x)\in[0,T)\times \partial B_m$},\\
u(T,x)=g(T,x),\quad \text{for $x\in B_m$}.
\end{array}
\right.
\end{align}
For the existence of a solution $u^{\eps,\delta;\gamma}_m$ to \eqref{eq:penalisedb} the key gradient bounds in \cite[Sec.\ 4.3]{bovo2022variational} can be recovered in our set-up with minor adjustments. In particular, \cite[Prop.\ 4.9]{bovo2022variational} holds by replacing \cite[Eq. (4.38)]{bovo2022variational} with
\begin{align}\label{eq:gbd0}
-2\langle \nabla w^n,\nabla(|\nabla u^n|_\gamma-(f^\gamma)^2)\rangle\leq 2\lambda |\nabla u^n|^2_\gamma+\tilde{R}_n,
\end{align}
where $w^n$ and $u^n$ are two proxies of $u^{\eps,\delta;\gamma}_m$ defined in \cite[Lemma 4.7]{bovo2022variational} (with $w^n,u^n\to u^{\eps,\delta;\gamma}_m$ for $n\to\infty$), $\lambda$ is an arbitrary constant and $\tilde R_n$ is a remainder that vanishes when $n\to 0$.
Since $|\nabla u^n|^2_\gamma\geq \gamma|\nabla u^n|_d^2$ then, using \eqref{eq:gbd0}, \cite[Eq. (4.37)]{bovo2022variational} becomes
\begin{align*}
0\leq (C_1-\lambda\gamma)|\nabla u|^2_d +C_2 +\lambda r M_1+R_n+\tilde{R}_n,
\end{align*}
where $C_1, C_2, M_1$ are the same constants as in the original paper and $R_n$ is another vanishing remainder.
The rest of the proof is the same, up to choosing $\lambda=\gamma^{-1}(C_1+1)$.

The proof of Proposition \cite[Prop.\ 5.1]{bovo2022variational}, which gives another gradient bound for $u^{\eps,\delta;\gamma}_m$ (uniformly in $m$) requires analogous changes. In particular, \cite[Eq. (5.10)]{bovo2022variational} becomes
\begin{align*}
\xi\langle\nabla w^n,\nabla\big(|\nabla u^n|^2_\gamma-(f^\gamma)^2\big)\rangle\ge &\,\lambda|\nabla u|^2_\gamma-|\nabla u|^3_d|\nabla \xi|_d-\xi \hat R_n\ge\lambda\gamma|\nabla u|_d^2-|\nabla u|^3_d|\nabla \xi|_d-\xi \hat R_n,
\end{align*}
where $\xi$ is a cut-off function \eqref{eq:xi} supported on $B_{m_0}$ for fixed $m_0<m$, and $\hat R_n$ is a vanishing remainder. The rest of the proof continues as in the original paper. Care is only needed to replace $\bar \lambda$ in \cite{bovo2022variational} with $\gamma^{-1}\bar \lambda$.

Thanks to the gradient bounds, and following the arguments from \cite{bovo2022variational}, we show the existence and uniquenss of solution to \eqref{eq:penalisedb}. Then, letting $m\to\infty$, we obtain the unique solution $u^{\eps,\delta;\gamma}$ of \eqref{eq:penalised}; for this convergence we need a bound on the penalisation term $\psi_\eps(|\nabla u^{\eps,\delta;\gamma}|^2_\gamma-(f^\gamma)^2)$. For that we argue as in \cite[Lem.\ 5.8]{bovo2022variational}. Its proof still holds upon noticing that the last term of \cite[Eq.\ (5.34)]{bovo2022variational} now reads 
\[
\frac{1}{2}\xi\psi'_\eps(\bar\zeta_n)\sum_{i,j=1}^da_{i,j}\Big(2\langle \nabla^0 \partial_{x_i}w^n,\nabla^0 \partial_{x_j}w^n\rangle+2\gamma\langle \nabla^1 \partial_{x_i}w^n,\nabla^1 \partial_{x_j}w^n\rangle\Big)
\]
where $\bar \zeta_n$ is the same as in \cite{bovo2022variational} and $w_n$ is a proxy for $u^{\eps,\delta;\gamma}$.
The sum can be bounded from below by
\begin{align*}
&\sum_{i,j=1}^da_{ij}\Big(2\langle \nabla^0 \partial_{x_i}w^n,\nabla^0\partial_{x_j} w^n\rangle+2\gamma\langle \nabla^1 \partial_{x_i}w^n,\nabla^1 \partial_{x_j}w^n\rangle\Big)\\
&=2\sum_{k=1}^{d_0}\sum_{i,j=1}^da_{ij} (\partial_{x_ix_k}w^n)( \partial_{x_jx_k}w^n)+2\gamma\sum_{k=d_0+1}^{d}\sum_{i,j=1}^da_{ij} (\partial_{x_ix_k}w^n)( \partial_{x_jx_k}w^n)\\
&\geq 2\sum_{k=1}^{d_0}|\nabla (\partial_{x_k}w^n)|_d^2+2\gamma \sum_{k=d_0+1}^{d}|\nabla (\partial_{x_k}w^n)|_d^2\geq 2\gamma |D^2 w^n|_{d\times d}^2.
\end{align*}

Moreover, the first equation in \cite[p.~28]{bovo2022variational} becomes 
\begin{align*}
\big\langle a\nabla \xi,&\,\nabla (\cH^\gamma(\nabla w^n)-(f^\gamma)^2) \big\rangle \geq\,-\xi\frac{\theta\gamma}{4}|D^2w^n|_{d\times d}^2-\frac{16}{\theta\gamma}\bar a_m^2 d^4 C_0|\nabla w^n|_d^2
\end{align*}
and the first equation in \cite[p.~30]{bovo2022variational} becomes
\begin{align*}
\sum_{k=1}^d \partial_{x_k} w^n \mathcal{L}_{x_k} w^n\le \tfrac{\theta\gamma}{8}|D^2 w^n|^2_{d\times d}+\frac{C_1}{\gamma}(N_1+1)^2,
\end{align*}
where $\cL_{x_k}$ is the second order differential operator defined as
\[
\big(\cL_{x_k}\varphi\big)(x)=\frac12\mathrm{tr}\big(\partial_{x_k}a(x)D^2\varphi(x)\big)+\langle\partial_{x_k}b(x),\nabla\varphi(x)\rangle,\quad\text{for $\varphi\in C^\infty(\R^d)$.}
\]
The constants $C_0,C_1,N_1,\bar a_m$ are independent of $\eps$ and $\delta$ and they are defined in \cite{bovo2022variational}. 
The rest of the proof continues as in the original paper.

From a probabilistic perspective, the solutions of the penalised problems admit representations in terms of 2-player, zero-sum stochastic differential games. Those games depend on a Hamiltonian function \cite[Eq.\ (4.4)]{bovo2022variational}, which in our setting reads
\begin{align}\label{eq:Ham}
H^{\eps,\gamma}(t,x,y)\coloneqq \sup_{p\in\R^d}\big\{\langle y,p\rangle_\gamma-\psi_{\eps}\big(|p|^2_\gamma-(f^\gamma(t))^2\big)\big\}.
\end{align}
For the problems on bounded domains the results in \cite[Prop. 4.1 and 4.3]{bovo2022variational} continue to hold. The only observation we need is that the first-order condition for $H^{\eps,\gamma}$ is the same as the one for $H^{\eps}_m$ used in the proof of \cite[Prop.\ 4.1]{bovo2022variational}, i.e.,
\begin{align*}
y_i&=\psi_{\varepsilon}'\big(|p|^2_\gamma- (f^\gamma(t))^2\big)2p_i\:\quad\text{for $i=1,\ldots, d_0$},\\
\gamma y_i&=\psi_{\varepsilon}'\big(|p|^2_\gamma- (f^\gamma(t))^2\big)2\gamma p_i\quad\text{for $i=d_0+1,\ldots, d$}.
\end{align*}
In vector notation, we have $y=\psi_{\varepsilon}'(|p|^2_\gamma- (f^\gamma(t))^2)2p$, which is precisely the same as in the paragraph above \cite[Eq.~(4.15)]{bovo2022variational}. Similarly, \cite[Prop. 5.4 and 5.5]{bovo2022variational} continue to hold by the same argument.

Another very small tweak affects \cite[Eq. (5.18)]{bovo2022variational} when $H^{\eps}$ is replaced by $H^{\eps,\gamma}$. Taking $p=(\eps /2)y$ in \eqref{eq:Ham} and using $\psi_\eps(z) \le z / \eps$ yields
\begin{align*}
H^{\varepsilon,\gamma}(t,x,y)\geq &\,\tfrac{\eps}{2}|y|_{\gamma}^2-\psi_\eps\big(|\tfrac{\eps}{2}y|^2_\gamma-(f^\gamma(t))^2\big)\geq \tfrac{\eps}{2}|y|_{\gamma}^2-\psi_\eps\big(|\tfrac{\eps}{2}y|^2_\gamma\big)
\geq \tfrac{\eps\gamma}{4}|y|_d^2.\notag
\end{align*}
Thus \cite[Eq. (5.26)]{bovo2022variational} holds with $\eps$ replaced by $\gamma\eps$.

All the remaining arguments from \cite{bovo2022variational} remain unaltered. \hfill $\square$

\section{}
We give below an extension of the result in \cite[Lemma 5.1]{deangelis2019numerical}. 
\begin{lemma}\label{lem:DGI_jumps}
Let $X$ be a real valued c\`adl\`ag semimartingale with jumps of bounded variation and let $L_t^0(X)$ be its local time at $0$ in the time-interval $[0,t]$. Then, for any $\eps\in(0,1)$ we have
\begin{align*}
\E[L_t^0(X)] \leq &4\eps -2\E\Big[\int_0^t\!\big(\mathds{1}_{\{X_s\in[0,\eps)\}}+\mathds{1}_{\{X_s\geq\eps\}}e^{1-\frac{X_s}{\eps}}\big)\,\ud X_s^c\Big]\\
&+\frac{1}{\eps}\E\Big[\int_0^t\!\mathds{1}_{\{X_s>\eps\}}e^{1-\frac{X_s}{\eps}}\,\ud \langle X\rangle_s^c\Big]+ 2 \E\Big[\sum_{0\leq s\leq t}|\Delta X_s|\Big],
\end{align*}
where $X_s^c$ and $\langle X\rangle_s^c$ are the continuous parts of $X$ and of the quadratic variation of $X$, respectively, and $\Delta X_s\coloneqq X_s-X_{s-}$.
\end{lemma}
\begin{proof}
For $\eps\in(0,1)$,  let
\begin{equation*}
g_\eps(y)=0\cdot\mathds{1}_{\{y<0\}}+y\mathds{1}_{\{0\leq y<\eps\}}+\eps\big(2-e^{1-\frac{y}{\eps}}\big)\mathds{1}_{\{y\geq\eps\}}, \qquad  \text{for }y\in\R.
\end{equation*} 
Following arguments from \cite[Lemma 5.1]{deangelis2019numerical} we have that $g_\eps\in C^1(\R\setminus\{0\})$, it is semi-concave, i.e., $y\mapsto g_\eps(y)-y^2$ is concave. Moreover, $g_\eps$ is such that
\begin{align*}
&0\leq g_\eps(y)\leq 2\eps,\quad\text{for $y\in\R$;}\\
&g'_\eps(y)=\mathds{1}_{\{0\leq y\leq \eps\}}+e^{1-\frac{y}{\eps}}\mathds{1}_{\{y\geq \eps\}},\quad\text{for $y\in\R$;}\\
&g''_\eps(y)=0,\quad\text{for $y\in(-\infty,0)\cup(0,\eps)$;}\\
&g''_\eps(y)=-\eps^{-1}e^{1-\frac{y}{\eps}},\quad\text{for $y>\eps$.}
\end{align*}
Applying \cite[Thm.\ IV.70 and Cor.\ IV.70.1]{protter2005stochastic} to $g_\eps(X_t)$ we get 
\begin{align*}
g_\eps(X_t)-g_\eps(X_0)=&\,\int_0^t g_\eps'(X_{s-})\,\ud X_s^c+\frac{1}{2}\int_0^t g_\eps''(X_s)\mathds{1}_{\{X_s\neq0\}\cap\{X_s\neq\eps\}}\,\ud \langle X\rangle_s^c\\
&\,+\frac12 L_t^0(X)+\sum_{0\leq s\leq t}\big(g_{\eps}(X_s)-g_{\eps}(X_{s-})\big).
\end{align*}
Rearranging terms and multiplying by 2, using $|g_{\eps}(X_s)-g_{\eps}(X_{s-})|\leq |X_s-X_{s-}|=|\Delta X_s|$ by Lipschitz property of $g_\eps$ and that $X$ has jumps of finite variation, we get
\begin{align*}
L_t^0(X)\leq\, 4\eps-2\int_0^t\! g_\eps'(X_{s-})\,\ud X_s^c-\int_0^t \!g_\eps''(X_s)\mathds{1}_{\{X_s\neq0\}\cap\{X_s\neq\eps\}}\,\ud \langle X\rangle_s^c+2\sum_{0\leq s\leq t}|\Delta X_s|.
\end{align*}
Using the properties of $g_\eps$ listed above and applying expectation we obtain the desired result.
\end{proof}

\bibliographystyle{plain}
\bibliography{Bibliography}

\begin{thebibliography}{10}

\bibitem{bandini2022optimal}
E.~Bandini, T.~De~Angelis, G.~Ferrari, and F.~Gozzi.
\newblock Optimal dividend payout under stochastic discounting.
\newblock {\em Math.\ Finance}, 32(2):627--677, 2022.

\bibitem{bayraktar2019controller}
E.~Bayraktar and J.~Li.
\newblock On the controller-stopper problems with controlled jumps.
\newblock {\em Appl.\ Math.\ Optim.}, 80(1):195--222, 2019.

\bibitem{bensoussan1974nonlinear}
A~Bensoussan and A.~Friedman.
\newblock Nonlinear variational inequalities and differential games with
  stopping times.
\newblock {\em J.\ Funct.\ Anal.}, 16(3):305--352, 1974.

\bibitem{bovo2022variational}
A.~Bovo, T.~De~Angelis, and E.~Issoglio.
\newblock Variational inequalities on unbounded domains for zero-sum
  singular-controller vs. stopper games.
\newblock {\em To appear in Math.\ Oper.\ Res. (arXiv:2203.06247)}, 2022.

\bibitem{cai2022american}
C.~Cai, T.~De~Angelis, and J.~Palczewski.
\newblock The {A}merican put with finite-time maturity and stochastic interest
  rate.
\newblock {\em Math.\ Finance}, 32(4):1170--1213, 2022.

\bibitem{chiarolla2016hilbert}
M.B. Chiarolla and T.~De~Angelis.
\newblock Optimal stopping of a {H}ilbert space valued diffusion: An infinite
  dimensional variational inequality.
\newblock {\em Appl.\ Math.\ Optim.}, 73(2):271--312, 2016.

\bibitem{de2017optimal}
T.~De~Angelis, S.~Federico, and G.~Ferrari.
\newblock Optimal boundary surface for irreversible investment with stochastic
  costs.
\newblock {\em Math.\ Oper.\ Res.}, 42(4):1135--1161, 2017.

\bibitem{de2015nonconvex}
T.~De~Angelis, G.~Ferrari, and J.~Moriarty.
\newblock A nonconvex singular stochastic control problem and its related
  optimal stopping boundaries.
\newblock {\em SIAM J.\ Control Optim.}, 53(3):1199--1223, 2015.

\bibitem{deangelis2019solvable}
T.~De~Angelis, G.~Ferrari, and J.~Moriarty.
\newblock A solvable 2-dimensional degenerate singular stochastic control
  problem with nonconvex costs.
\newblock {\em Math.\ Oper.\ Res.}, 44(2):512--531, 2019.

\bibitem{deangelis2019numerical}
T.~De~Angelis, M.~Germain, and E.~Issoglio.
\newblock A numerical scheme for stochastic differential equations with
  distributional drift.
\newblock {\em Stoch.\ Process.\ Appl.}, (154):55--90, 2022.

\bibitem{evans10}
{L.C.} Evans.
\newblock {\em Partial differential equations}, volume~19 of {\em Graduate
  Studies in Mathematics}.
\newblock American Mathematical Society, Providence, RI, Providence, R.I.,
  second edition, 2010.

\bibitem{federico2021singular}
S.~Federico, G.~Ferrari, and P.~Schuhmann.
\newblock Singular control of the drift of a {B}rownian system.
\newblock {\em Appl.\ Math.\ Optim.}, 84:561--590, 2021.

\bibitem{federico2014characterization}
S.~Federico and H.~Pham.
\newblock Characterization of the optimal boundaries in reversible investment
  problems.
\newblock {\em SIAM J.\ Control Optim.}, 52(4):2180--2223, 2014.

\bibitem{ferrari2015integral}
G.~Ferrari.
\newblock On an integral equation for the free-boundary of stochastic,
  irreversible investment problems.
\newblock {\em Ann.\ Appl.\ Probab.}, 25(1):150--176, 2015.

\bibitem{fleming2006controlled}
W.H. Fleming and H.M. Soner.
\newblock {\em Controlled {M}arkov processes and viscosity solutions},
  volume~25 of {\em Stochastic Modelling and Applied Probability}.
\newblock Springer, New York, second edition, 2006.

\bibitem{friedman2008partial}
A.~Friedman.
\newblock {\em Partial differential equations of parabolic type}.
\newblock Courier Dover Publications, 2008.

\bibitem{guo2009class}
X.~Guo and P.~Tomecek.
\newblock A class of singular control problems and the smooth fit principle.
\newblock {\em SIAM J.\ Control Optim.}, 47(6):3076--3099, 2009.

\bibitem{hamadene2006stochastic}
S.~Hamadène.
\newblock Mixed zero-sum stochastic differential game and {American} game
  options.
\newblock {\em SIAM J.\ Control Optim.}, 45(2):496--518, 2006.

\bibitem{hernandez2015zero}
D.~Hernandez-Hernandez, {R.S.} Simon, and M.~Zervos.
\newblock A zero-sum game between a singular stochastic controller and a
  discretionary stopper.
\newblock {\em Ann.\ Appl.\ Probab.}, 25(1):46--80, 2015.

\bibitem{hernandez2015zsgsingular}
D.~Hernandez-Hernandez and K.~Yamazaki.
\newblock Games of singular control and stopping driven by spectrally one-sided
  {L}évy processes.
\newblock {\em Stoch.\ Process.\ Appl.}, 125(1):1--38, 2015.

\bibitem{karatzas2001controller}
I.~Karatzas and W.D. Sudderth.
\newblock The controller-and-stopper game for a linear diffusion.
\newblock {\em Ann.\ Probab.}, 29(3):1111--1127, 2001.

\bibitem{karatzas2008martingale}
I.~Karatzas and I.-M. Zamfirescu.
\newblock Martingale approach to stochastic differential games of control and
  stopping.
\newblock {\em Ann.\ Probab.}, 36(4):1495--1527, 2008.

\bibitem{krylov2008lectures}
N.V. Krylov.
\newblock {\em Lectures on elliptic and parabolic equations in {S}obolev
  spaces}, volume~96 of {\em Graduate Studies in Mathematics}.
\newblock American Mathematical Society, Providence, RI, 2008.

\bibitem{krylov1980controlled}
N.V. Krylov.
\newblock {\em Controlled diffusion processes}, volume~14 of {\em Stochastic
  Modelling and Applied Probability}.
\newblock Springer-Verlag, Berlin, 2009.
\newblock Translated from the 1977 Russian original by A.B. Aries, Reprint of
  the 1980 edition.

\bibitem{lokka2011model}
A.~L{\o}kka and M.~Zervos.
\newblock A model for the long-term optimal capacity level of an investment
  project.
\newblock {\em Int.\ J.\ Th.\ Appl.\ Finance}, 14(02):187--196, 2011.

\bibitem{lokka2013long}
A.~L{\o}kka and M.~Zervos.
\newblock Long-term optimal investment strategies in the presence of adjustment
  costs.
\newblock {\em SIAM J.\ Control Optim.}, 51(2):996--1034, 2013.

\bibitem{merhi2007model}
A.~Merhi and M.~Zervos.
\newblock A model for reversible investment capacity expansion.
\newblock {\em SIAM J. Control Optim.}, 46(3):839--876, 2007.

\bibitem{protter2005stochastic}
P.E. Protter.
\newblock {\em Stochastic integration and differential equations}, volume~21 of
  {\em Stochastic Modelling and Applied Probability}.
\newblock Springer-Verlag, Berlin, second edition, 2005.

\bibitem{soner1991free}
H.M. Soner and S.E. Shreve.
\newblock A free boundary problem related to singular stochastic control: the
  parabolic case.
\newblock {\em Comm.\ Partial Differential Equations}, 16(2-3):373--424, 1991.

\end{thebibliography}

\end{document}